\documentclass[10pt,reqno]{article}

\usepackage{amsmath, amsthm, amscd, amsfonts, amssymb, graphicx, color, mathrsfs,mathtools}
\usepackage{lineno}
\usepackage{authblk}
\usepackage{dsfont}
\usepackage[bookmarksnumbered, colorlinks, plainpages]{hyperref}
\hypersetup{colorlinks=true,linkcolor=red, anchorcolor=green, citecolor=cyan, urlcolor=red, filecolor=magenta, pdftoolbar=true}
\modulolinenumbers[5]

\newtheorem{thm}{Theorem}[section]
\newtheorem{lemma}[thm]{Lemma}
\newtheorem{pro}[thm]{Proposition}

\theoremstyle{definition}
\newtheorem{defn}[thm]{Definition}
\newtheorem{example}[thm]{Example}

\newtheorem{hyp}[thm]{Hypothesis}

\theoremstyle{remark}
\newtheorem{rmk}[thm]{Remark}

\numberwithin{equation}{section}

\renewcommand{\hat}[1]{\widehat{#1}}
\renewcommand{\tilde}[1]{\widetilde{#1}}

\newcommand{\N}{\mathbb{N}}

\newcommand{\R}{\mathbb{R}}
\newcommand{\C}{\mathbb{C}}
\newcommand{\elle}{\operatorname{L}}

\newcommand{\fcon}{\operatorname{\mathscr{FC}}}

\DeclareMathOperator{\sign}{\operatorname{sign}}

\title{Analyticity of nonsymmetric Ornstein-Uhlenbeck semigroup with respect to a weighted Gaussian measure}

\author{{D. Addona}
\thanks{email: davide.addona@unimib.it}}
\affil{Department of Mathematics and applications\\
University of Milano Bicocca\\
via Cozzi 55, 20125 Milano, Italy}
\date{}
\providecommand{\keywords}[1]{{\textit{Keywords}:} #1}
\providecommand{\subjclass}[1]{{\textit{SubjClass}[2000]:} #1}

\begin{document}

\frenchspacing

\maketitle 
%\address[D. Addona]{Dipartimento di Matematica e Informatica, Universit\`a di Milano Bicocca, via Cozzi, 55, 20126 Milano, Italy}
%\email{\textcolor[rgb]{0.00,0.00,0.84}{davide.addona@unimib.it}}

\begin{abstract}
In this paper we show that the realization in $\elle^p(X,\nu_\infty)$ of a nonsymmetric Ornstein-Uhlenbeck operator $L_p$ is sectorial for any $p\in(1,+\infty)$ and we provide an explicit sector of analyticity. Here, $(X,\mu_\infty,H_\infty)$ is an abstract Wiener space, i.e., $X$ is a separable Banach space, $\mu_\infty$ is a centred non degenerate Gaussian measure on $X$ and $H_\infty$ is the associated Cameron-Martin space. Further, $\nu_\infty$ is a weighted Gaussian measure, that is, $\nu_\infty=e^{-U}\mu_\infty$ where $U$ is a convex function which satisfies some minimal conditions. Our results strongly rely on the theory of nonsymmetric Dirichlet forms and on the divergence form of the realization of $L_2$ in $\elle^2(X,\nu_\infty)$.
\end{abstract}

\vspace{5mm}

\vspace{2mm}

\section{Introduction}
In this paper we prove that the realization in $\elle^p(X,\nu_\infty)$ of the nonsymmetric perturbed Ornstein-Uhlenbeck operator $L_p$ defined on smooth functions $f$ by
\begin{align}
\label{pert_O-U}
L_pf(x)=\frac12{\rm Tr}[D^2f(x)]_H+\langle x,A^*Df(x)\rangle_{X\times X^*}+[BD_Hf(x),D_HU(x)]_H, \quad x\in X,
\end{align}
where $U$ is a suitable function (see Hypothesis \ref{ventola}), is sectorial in $\elle^2(X,\nu_\infty)$ and we provide an explicit sector of analyticity.
%A family of linear bounded operators $(T(t))_{t\geq0}$, $T(t)\in\mathcal L(X)$ (where $X$ is a Banach space and $\mathcal L(X)$ is the set of linear bounded operators on $X$) for any $t\geq0$, which satisfies the functional relation
%\begin{align*}
%T(t+s)=T(t)T(s), \quad s,t\geq0, \qquad T(0)=Id_X,
%\end{align*}
%is called {\it semigroup} and naturally arises in the study of autonomous deterministic systems, both in finite and infinite dimension. An important class of semigroups consists of {\it analytic semigroups}, so called since the map $t\mapsto T(t)$ is analytic from $(0,+\infty)$ into $\mathcal L(X^*)$. This class of semigroups enjoys nice properites, see e.g. \cite{LU95} for a systematic treatment of the basis theory of analytic semigroups, abstract parabolic equations in general Banach spaces and their applications to parabolic PDE's.

In finite dimension, the Ornstein-Uhlenbeck operator is the uniformly elliptic second order differential operator $\mathcal L$ defined on smooth functions $\varphi$ by
\begin{align*}
%\label{OU_fin_dim}
\mathcal L\varphi(\xi)=\sum_{i,j=1}^n q_{ij} D^2_{ij}\varphi(\xi)+\sum_{i,j=1}^na_{ij}\xi_jD_i\varphi(\xi), \quad \xi\in \R^n,
\end{align*}
where $Q=(q_{ij})_{i,j=1}^n$ is a positive definite matrix and $A=(a_{ij})_{i,j=1}^n$. It is well known (see \cite{M01,MP04}) that $\mathcal L$ may fail to generate an analytic semigroup on $\elle^p(\R^n)$. 
The additional assumption $\sigma(A)\subseteq\{z\in\C:{\rm Re}z<0\}$ implies that the integral
\begin{align*}
Q_\infty:=\int_0^{+\infty}e^{tA}Qe^{tA^*}dt,
\end{align*}
is well defined. The centred Gaussian measure $\mu_\infty$ with covariance $Q_\infty$ is an invariant measure for $\mathcal L$, i.e.,  
\begin{align*}
\int_{\R^n}\mathcal Lf d\mu_\infty =0, \quad f\in D(\mathcal L).
\end{align*}
$\mathcal L$ behaves well in $\elle^p(\R^n,\mu_\infty)$. Indeed, the realization $\mathcal L_p$ of $\mathcal L$ in $\elle^p(\R^n,\mu_\infty)$ generates an analytic semigroup for any $p\in(1,+\infty)$. In \cite{CFMP05} the authors explicitly provide a sector 
\begin{align}
\label{sector_p}
\Sigma_{\theta_p}:=\{re^{i\phi}\in\C:\ r>0, \ |\phi|\leq \theta_p\},
\end{align}
where $\theta_p\in(0,\pi/2)$ is an angle which depends on $Q,A$ and $p$, such that $\mathcal L_p$ is sectorial in $\Sigma_{\theta_p}$. This sector is optimal, in the sense that if $\theta\in(0,\pi/2)$ is an angle such that $\mathcal L_p$ is sectorial in $\Sigma_\theta$, then $\theta\leq \theta_p$. In \cite{CFMP06} the same authors extend this result to nonsymmetric sub-Markovian semigroups.

In infinite dimension the situation is much more complicated. We consider an abstract Wiener spaces $(X,\mu_\infty,H_\infty)$, where $X$ is a separable Banach space, $\mu_\infty$ is a centred nondegenerate Gaussian measure on $X$ and $H_\infty$ is the associated Cameron-Martin space (see e.g. \cite{Bog98}). It is well known that $H_\infty\subseteq X$ is a Hilbert space with inner product $[\cdot,\cdot]_{H_\infty}$. Let us denote by $Q_\infty:X^*\rightarrow X$ the covariance operator of $\mu_\infty$. In this setting, the definition of the Ornstein-Uhlenbeck operator can be given in terms of bilinear forms: given $f,g\in C_b^1(X)$ we set
\begin{align*}
\mathcal E(f,g):=\int_X[D_{H_\infty}f,D_{H_\infty}g]_{H_\infty}d\mu_\infty,
\end{align*}
where $D_{H_\infty}=Q_\infty D$ is the gradient along the directions of $H_{\infty}$. Following \cite[Chapter 1]{MR92} it follows that there exists an operator $\mathscr L_2:D(\mathscr L_2)\subset \elle^2(X,\mu_\infty)\rightarrow X$ such that for any $f\in D(\mathscr L_2)$ and any $g\in C_b^1(X)$ we have
\begin{align*}
\mathcal E(f,g)=-\int_X \mathscr L_2f gd\mu_\infty .
\end{align*}
The operator $\mathscr L_2$ is self-adjoint and it generates an analytic contraction $C_0$-semigroup on $\elle^2(X,\mu_\infty)$. Moreover, if $f=\varphi(x_1^*,\ldots,x_n^*)$ for some smooth function $\varphi$ and $x_i^*\in X^*$, $i=1,\ldots,n$, then the operator $\mathscr L_2$ reads as
\begin{align*}
\mathscr  L_2f:=\sum_{i,j=1}^nq^0_{ij}\frac{\partial^2\varphi}{\partial \xi_i \partial \xi_j}-\sum_{,i=1}^nx_i^*\frac{\partial \varphi}{\partial \xi_i},
\end{align*}
where $q^0_{ij}=\langle Q_\infty x^*_j,x^*_i\rangle_{X\times X^*}$. In \cite{GV03} the authors provide a generalization of $\mathscr L_2$, defining the Wiener space $(X,\mu_\infty,H_\infty)$ as follows. They consider two operators $Q:X^*\rightarrow X$ and $A:D(A)\subset X\rightarrow X$ such that $Q$ is a linear, bounded, nonnegative and symmetric operator (see Hypothesis \ref{ipo_1}) and $A$ is the infinitesimal generator of a strongly continuous semigroup. Let us denote by $(e^{tA})_{t\geq0}$ the semigroup generated by $A$. They assume that the integral
\begin{align*}
\int_0^\infty e^{tA}Qe^{tA^*}dt,
\end{align*}
with values in $\mathcal L(X^*;X)$, exists as a Pettis integral and the operator $Q_\infty:X^*\rightarrow X$ defined by
\begin{align*}
Q_\infty x^*:= \int_0^\infty e^{tA}Qe^{tA^*}dt\ \!x^*,
\end{align*}
is the covariance operator of the Gaussian measure $\mu_\infty$. In such a way they can define the Reproducing Kernel Hilbert Space $H$ associated to $Q$, and they prove the closability of a gradient operator $D_H=QD$. Thanks to a stochastic representation, the authors define a semigroup $P(t)$ and its infinitesimal generator $\mathbb L$ on $\elle^p(X,\mu_\infty)$ which on smooth functions $f$ (with $f=\varphi(x_1^*,\ldots,x_n^*)$, for some $\varphi\in C_b^2(\R^n)$, $n\in\N$ and $x_i^*\in D(A^*)$, $i=1,\ldots,n$) reads as
\begin{align*}
\mathbb Lf:=\sum_{i,j=1}^n \tilde q_{ij}\frac{\partial^2 \varphi}{\partial \xi_i\partial\xi_j}+\sum_{i=1}^n A^*x^*_i \frac{\partial \varphi}{\partial\xi_i},
\end{align*}
with $\tilde q_{ij}=\langle Qx^*_j,x^*_i\rangle_{X\times X^*}$. From the results in \cite{GK98}, the authors deduce that the set
\begin{align*}
\mathscr F_0:=\{f\in \mathscr F:\langle \cdot, A^*Df\rangle_{X\times X^*}\in C_b(X)\},
\end{align*}
is a core for $\mathbb L$. Here $\mathscr F$ is the set of functions $f\in C_b^2(X)$ such that there exists $\varphi\in C_b^2(\R^n)$ and $x_1^*,\ldots,x_n^*\in D(A^*)$ such that $f(x)=\varphi(\langle x,x_1^*\rangle _{X\times X^*}, \ldots, \langle x,x_1^*\rangle _{X\times X^*})$ for any $x\in X$. Finally, arguing as in \cite{G99}, the authors show different characterizations of the analyticity of $P(t)$. In particular, they prove that $P(t)$ is analytic in $\elle^2(X,\mu_\infty)$ if and only if $Q_\infty A^*x^*\in H$ for any $x^*\in D(A^*)$ and there exists a positive constant $c$ such that
\begin{align*}
|Q_\infty A^* x^*|_H\leq c|Qx^*|, \quad x^*\in D(A^*).
\end{align*}

This characterization is the starting point of \cite{MV07}, where the authors generalize the results in \cite{CFMP05} to the infinite dimensional case.
% without any assumption on the nondegeneracy of $Q$. 
To begin with, they prove that the operator $B\in\mathcal L(H)$, which is the extension of $Q_\infty A^*$ to the whole $H$, satisfies $B+B^*=-Id_H$. Let
\begin{align*}
\mathcal E_B(u,v):=-\int_X[BD_Hu,D_Hv]_Hd\mu_\infty,
\end{align*}
on $u,v\in C_b^1(X)$, and let $\widetilde {\mathbb L}:D(\widetilde {\mathbb L})\subset X\rightarrow X$ be the operator associated to $\mathcal E_B$ in $\elle^2(X,\mu_\infty)$ in the sense of \cite[Chapter 1]{MR92}, i.e., 
\begin{align*}
\mathcal E_B(u,v)=-\int_X\mathbb Luvd\mu_\infty,
\end{align*}
for any $u\in D(\widetilde {\mathbb L})$ and $v\in C_b^1(X)$. The authors show that $\widetilde {\mathbb L}=\mathbb L$, where $\mathbb L$ is the infinitesimal generator of $P(t)$. 
%This implies that, if we denote by $D_H^*$ the adjoint operator of $D_H$ in $\elle^2(X,\mu_\infty)$, then $\mathbb L=D_HBD_H$, and by means of the divergence form of $\mathbb L$ the authors avoid the nondegeneracy assumption on $Q$. 
By means of the the numerical range theorem (see \cite{GR97}) the authors prove that for any $p\in(1,+\infty)$ the semigroup $P(t)$ is analytic in $\elle^p(X,\mu_\infty)$ with sector of analiticity $\Sigma_{\theta_p}$ defined in \eqref{sector_p}. Also in this case, this sector is optimal.
We remark that, differently from $\mathscr L_2$, in general the operator $\mathbb L$ is not self-adjoint and therefore it is not possible to use the theory of self-adjoint operators to prove the analyticity of $\mathbb L$.  

In this paper we consider the operator $L_2$ associated in $\elle^2(X,\nu_\infty)$ to the nonsymmetric bilinear form
%where $U$ is a suitable function (see e.g. \cite{CF16,DPL14,Fer15}), and we prove that for any $p\in(1,+\infty)$ the realization $L_p$ in $\elle^p(X,\nu_\infty)$ of the Ornstein-Uhlenbeck operator is sectorial with same sector $\Sigma_{\theta_p}$ defined in \eqref{sector_p}. Again, we consider the bilinear form 
\begin{align*}
\mathcal E^\nu_B(u,v):=-\int_X[BD_Hu,D_Hv]_Hd\nu_\infty,
\end{align*}
in the sense of \cite{MR92}, where
\begin{align*}
\nu_\infty:=e^{-U}\mu_\infty.
\end{align*}
On smooth functions the operator $L_2$ has the form \eqref{pert_O-U}.
%Further, $L_2=D_H^*BD_H$, where $D_H^*$ denotes the adjoint operator of $D_H$ in $\elle^2(X,\nu_\infty)$. 
By taking advantage of the definition of $L_2$ and its adjoint operator $L_2^*$ in $\elle^2(X,\nu_\infty)$, we extend $L_2$ and the associated semigroup to $\elle^p(X,\nu_\infty)$, $p\in(1,+\infty)$. Finally, we prove that the semigroup associated to $L_p$ is analytic in $\elle^p(X,\nu_\infty)$ with sector of analyticity $\Sigma_{\theta_p}$, and we provide an example to which our results apply.

We stress that, at the best of our knowledge, in the case of perturbed Ornstein-Uhlenbeck operator no explicit core of $L_p$  is known. However, for $p\geq2$ we identify a set of smooth functions which allows us to overcome this difficulty, and a we obtain the desired result. In the case $p\in(1,2)$ we take advantage of the fact that $D(L_2)$ is a core for $L_p$.

It would be interesting to provide more examples to which apply our results and to understand some features of the covariance operator $Q_\infty$. Indeed, if one consider the classical Wiener space, i.e., the case $X=\elle^2(0,1)$, $Q$ as in \eqref{op_Q_ex} and $A=-Id$, then $Q_\infty=\frac12 Q$ and a function $f$ is an eigenvector of $Q$ with eigenvalue $\lambda$ if and only if $f$ solves on $(0,1)$ the problem
\begin{align*}
\lambda f''+f=0, \quad f(0)=0, \quad f'(1)=0.
\end{align*}
However, also in apparently friendly contexts the situation is far to be well understood. In the example which we provide in Section \ref{example} we have an explicit formula for $Q_\infty$, but we don't know how to get more informations on $Q_\infty$ and $L$. We devote these and other stimulating questions to future papers.

The paper is organized as follows. In Section \ref{ferragosto} we uniform the notations used in the symmetric and in the nonsymmetric case, which are different and sometimes may give rise to confusion and misunderstandings. Then, we prove that $D_H$ is closable on smooth functions in $\elle^p(X,\nu_\infty)$ for any $p\in(1,+\infty)$ and define the Sobolev spaces as the domain of the closure of $D_H$. 
%Further, as example \ref{example2} shows, the construction of the covariance $Q_\infty$ we deal with allows us to consider more directions along which differentiate and therefore we can define Sobolev spaces which are bigger with respect to the classical ones. 
Section \ref{nonsymm_OU_op_section} is devoted to define the nonsymmetric Ornstein-Uhlenbeck operator and semigroup in $\elle^p(X,\nu_\infty)$. At first, thanks to the theory of nonsymmetric Dirichlet forms, we provide the definition of the Ornstein-Uhlenbeck operator and semigroup in $\elle^2(X,\nu_\infty)$. Later, we extend both the operator $L_2$ and the semigroup $(T_2(t))_{t\geq0}$ to any $\elle^p(X,\nu_\infty)$, $p\in(1,+\infty)$. We conclude the section by showing the inclusion $D(L_q)\subset D(L_p)$ for any $p,q\in(1,+\infty)$ and $q>p$. These results allow us to overcome the fact that we don't know a core for $L_p$. 
%Indeed, \cite{MV07} widely use the fact that the set of cylindrical functions of the form $\varphi(x_1^*,\ldots,x_n^*)$ with $\varphi$ smooth function and $x_i^*\in D(A^*)$ for $i=1,\ldots,n$, is a core for $\mathbb L$. At the best of our knowledge, the presence of the weight function $U$ prevents of finding easily a core for $L_p$. However, we are able to find a good approximations for functions in $D(L_p)$ and 
In Section \ref{analyticity} we use the numerical range thorem to show that $L_p$ generates an analytic semigroup in $\elle^p(X,\nu_\infty)$ with sector $\Sigma_{\theta_p}$ for any $p\in(1+\infty)$. 
%We are not able to show the optimality of this sector since the techniques applied both in \cite{CFMP05} and in \cite{MV07} don't work in infinite dimension with a weighted Gaussian measure. 
Finally, in Section \ref{example} we provide a explicit example of operators $Q$ and $A$ and of function $U$ which satisfy our assumptions. 

\subsection{Notations}
Let $X$ be a separable Banach space. We denote by $\langle\cdot,\cdot\rangle_{X\times X^*}$ the duality, by $\|\cdot\|_X$ its norm and by $\|\cdot\|_{X^*}$ the norm of its dual. Further, for a general Banach space $V$ we denote by $\mathcal L(V)$ the space of linear bounded operators from $V$ onto $V$ endowed with the operator norm.
For any $k\in\N\cup\{\infty\}$ and any $n\in\N$ we denote by $C_b^k(\R^n)$ the continuous and bounded functions on $\R^n$ whose derivatives up to the order $k$ are continuous and bounded. We denote by $C_b^k(X)$ the set of Fr\'echet-differentiable functions on $X$ up to order $k$ with bounded Fr\'echet derivative.

Let $Y$ be a separable Hilbert space with inner product $\langle,\cdot,\cdot \rangle_{Y\times Y}$ and let $\gamma$ be a Borel measure on $X$. For any $p\in[1,+\infty)$ let us set 
\begin{align}
\|f\|_{\elle^p(X,\gamma;Y)}:=\left(\int_X|f(x)|_Y^p \gamma(dx)\right)^{1/p},
\end{align}
for any measurable function $:X\rightarrow Yf$. We denote by $\elle^p(X,\gamma;Y)$ the space of equivalence classes of Bochner integrable functions $f$ with $\|f\|_{\elle^p(X,\gamma;Y)}<+\infty$.

For any $y,z\in Y$ we denote by $y\otimes z:Y\times Y\rightarrow \R$ the map defined by 
\begin{align*}
(y\otimes z)(x,w)=\langle y,x\rangle_{Y\times Y}\langle z,w\rangle_{Y\times Y}, \quad x,w\in Y.
\end{align*}
 
\section{Preliminaries and Sobolev spaces}
\label{ferragosto}
We state the following assumptions on the operators $Q$ and $A$.
\begin{hyp}
\label{ipo_1}
\begin{itemize}
\item[(i)]
$Q:X^*\rightarrow X$ is a linear and bounded operator which is symmetric and nonnegative, i.e., 
\begin{align*}
\langle Q x^*, y^*\rangle_{X\times X^*}=\langle Q y^*, x^*\rangle_{X\times X^*}, \quad \langle Q x^*, x^*\rangle_{X\times X^*}\geq0, \quad \forall x^*,y^*\in X^*.
\end{align*}
%and  $\langle Q x^*, x^*\rangle_{X\times X^*} =0$ if and only if $x^*=0$. 
\item[(ii)]
$A:D(A)\subseteq X\rightarrow X$ is the infinitesimal generator of a strongly continuous contraction semigroup $\left(e^{tA}\right)_{t\geq0}$ on $X$.
\end{itemize}
\end{hyp}

The following definition shows that given a nonnegative and symmetric operator $F:X^*\rightarrow X$ we can define a Hilbert space $K\subset X$, which is called the Reproducing Kernel Hilbert Space associated to $F$.
\begin{defn}
\label{RKHS}
Let $F:X^*\rightarrow X$ be a linear, bounded, nonnegative and symmetric operator. On $FX^*$ we define the inner product $[Fx^*,Fy^*]_K:=\langle Fx^*,y^*\rangle_{X\times X^*}$ for any $x^*,y^*\in X^*$. We denote by $|Kx^*|^2_K:=\langle Fx^*,x^*\rangle_{X\times X^*}$ the associated norm.
We set $K:=\overline{FX^*}^{|\cdot|_K}\subset X$ and we call $K$ the Reproducing Kernel Hilbert Space (RKHS) associated with $F$.
\end{defn}

From \cite[Proposition 1.2]{VN98} the function $s\mapsto e^{sA}Qe^{sA^*}$ is strongly measurable and we may define, for any $t>0$, the nonnegative symmetric operator $Q_t\in \mathcal L(X^*;X)$ by
\begin{align*}
Q_t:=\int_0^t e^{sA}Qe^{sA^*}ds.
\end{align*}
Further, we denote by $H_t$ the Reproducing Kernel Hilbert Space associated to $Q_t$. We assume that the family of operators $(Q_t)_{t\geq0}$ satisfies the following hypotheses (see e.g. \cite[Sections 2 \& 6]{GV03}).
\begin{hyp}
\label{portafoglio}
\begin{enumerate}
\item [(i)]The operator $Q_t$ is the covariance operator of a centred Gaussian measure $\mu_t$ on $X$ for any $t>0$.
\item [(ii)] For any $x^*\in X^*$, there exists ${\rm weak}-\lim_{t\rightarrow+\infty}Q_t x^*=:Q_\infty x^* $ and $Q_\infty$ is the covariance operator of a centred nondegenerate Gaussian measure $\mu_\infty$.
\end{enumerate}
\end{hyp}

Hypothesis \ref{portafoglio}$(ii)$ implies that
\begin{align*}
\widehat{\mu_\infty}(f)=\exp\left(-\frac12 \langle Q_\infty f,f\rangle_{X\times X^*}\right), \quad f\in X^*.
\end{align*}
We follow \cite[Chapter 2]{Bog98} to construct the Cameron-Martin space $H_\infty$ associated to $\mu_\infty$, which gives the abstract Wiener space $(X,\mu_\infty,H_\infty)$. We conclude by showing that $H_\infty$ is the Reproducing Kernel Hilbert Space associated with $Q_\infty$.

From \cite[Fernique Theorem 2.8.5]{Bog98} it follows that $X^*\subset \elle^2(X,\mu_\infty)$, and we denote by $j:X^*\rightarrow \elle^2(X,\mu_\infty)$ the injection of $X^*$ in $\elle^2(X,\mu_\infty)$. From \cite[Theorem 2.2.4]{Bog98} we have
\begin{align}
\label{caratt_cova_inf}
\langle Q_\infty f,g\rangle_{X\times X^*}=\int_Xfgd\mu_\infty, \quad  f,g\in X^*.
\end{align}
We denote by $X_{\mu_\infty}^*$ the closure of $j(X^*)$ in $\elle^2(X,\mu_\infty)$ and we define $R:X^*_{\mu_\infty}\rightarrow (X^*)'$ by
\begin{align}
\label{op_R}
R(f)(g):=\int_Xfgd\mu_\infty, \quad f\in X_{\mu_\infty}^*, \ g\in X^*.
\end{align}
For any $f\in X^*_{\mu_\infty}$ the map $g\mapsto R(f)(g)$ is weakly$^*$-continuous on $X^*$, and therefore $R(X_{\mu_\infty}^*)\subset X$. We still denote by $R(f)$ the unique element $y\in X$ such that for any $g\in X^*$ we have $R(f)(g)=\langle y,g\rangle_{X\times X^*}$. Further, the injection $j$ is the adjoint operator of $R$. The Cameron-Martin space $H_\infty$ associated to $\mu_\infty$ is defined as follows (see e.g. \cite[Chapter 2, Section 2]{Bog98}):
\begin{align*}
|h|_{H_\infty}& :=\sup\left\{\langle h,\ell\rangle_{X\times X^*}:\ell\in X^*, \ R(\ell)(\ell)=\|R^*\ell\|^2_{\elle^2(X,\mu_\infty)}\leq 1\right\}, \\
H_\infty & :=\left\{h\in X:|h|_{H_\infty}<+\infty\right\}.
\end{align*}
From \cite[Lemma 2.4.1]{Bog98} it follows that $h\in H_\infty$ if and only if there exists $\hat h\in X_{\mu_\infty}^*$ such that $R(\hat h)=h$. $H_\infty$ is a Hilbert space if endowed with inner product
\begin{align}
\label{prod_scal_Hinf}
[h,k]_{H_\infty}=\langle \hat h,\hat k\rangle_{\elle^2(X,\mu_\infty)}, \quad h,k\in H_\infty.
\end{align}
We stress that for any $f\in X^*$, from \eqref{caratt_cova_inf} and \eqref{op_R} we have $Q_\infty f\in H_\infty$ and that $R(R^*f)=Q_\infty f$, i.e., $\widehat {Q_\infty f}=R^*f$. 
%Indeed, from \eqref{caratt_cova_inf} for any $g\in X^*$ such that $\|g\|_{\elle^2(X,\mu_\infty)}^2=R_{\mu_\infty}(g)(g)\leq 1$ we have
%\begin{align*}
%\langle Q_\infty f,g\rangle_{X\times X^*}=\int_X(R^*f)(R^*g)d\mu_\infty 
%\leq \|R^*f\|_{\elle^2(X,\mu_\infty)} \|R^*g\|_{\elle^2(X,\mu_\infty)}\leq \|R^*f\|_{\elle^2(X,\mu_\infty)}.
%\end{align*}
%From the definition of $H_\infty$ it follows that $Q_\infty f\in H_\infty$ and again from \eqref{caratt_cova_inf} and from the definition of $R_{\mu_\infty}$ we infer that
%\begin{align*}
%R_{\mu_\infty}(R^*f)(R*g)=\int_X(R^*f)(R^*g)d\mu_\infty=\langle Q_\infty f, g\rangle_{X\times X^*},
%\end{align*}
%which gives $R_{\mu_\infty}(R^*f)=Q_\infty f$ and therefore $\widehat{Q_\infty f}=R^*f$. 
Further, from \eqref{prod_scal_Hinf} we deduce that
\begin{align}
\label{inner_prod_inf_dual}
\langle Q_\infty f,g\rangle_{X\times X^*}=[Q_\infty f,Q_\infty g]_{H_\infty}, \quad f,g\in X^*.
\end{align}
We get the following characterization of $H_\infty$.
\begin{lemma}
\label{caratt_cameron_martin_inf}
$\displaystyle H_\infty=\overline{Q_\infty X^*}^{|\cdot|_{H_\infty}}$, that is, the Cameron-Martin space $H_\infty$ is the Reproducing Kernel Hilbert Space associated to $Q_\infty$.
\end{lemma}
\begin{proof}
The proof is quite simple but we provide it for reader's convenience. Let $h\in H_\infty$. Then, there exists $\hat h \in X_{\mu_\infty}^*$ such that $R_{\mu_\infty}(\hat h)=h$. In particular, there exists $(f_n)\subset X^*$ such that $R^*f_n\rightarrow \hat h$ in $\elle^2(X,\mu_\infty)$. We claim that $Q_\infty f_n\rightarrow h$ in $H_\infty$. Indeed, from \eqref{prod_scal_Hinf} and recalling that $\widehat {Q_\infty f_n}=R^*f_n$ for any $n\in\N$, it follows that
\begin{align*}
|Q_\infty f_n-h|_{H_\infty}^2
%= & [Q_\infty f_n-h,Q_\infty f_n-h]_{H_\infty}
= \int_X|R^*f_n-\hat h|^2d\mu_\infty\rightarrow0,\quad n\rightarrow+\infty.
\end{align*}
This means that $\displaystyle H_\infty\subseteq\overline{Q_\infty X^*}^{|\cdot|_{H_\infty}}$. The converse inclusion follows from analogous arguments.
\end{proof}

The continuous injection of $Q_\infty X^*$ into $X$ can be continuously extended to $H_\infty$. We denote by $i_\infty$ the extension of this injection. If we denote by $i_\infty^*:X^*\rightarrow H_\infty^*$ the adjoint operator and we identify $H_\infty^*$ with $H_\infty$ by means of the Riesz Representation Theorem, then $Q_\infty=i_\infty\circ i_\infty^*$. Indeed, for any $f,g\in X^*$ we have
\begin{align}
\label{prop_covariance_op}
\langle i_\infty\circ i_\infty^* f,g\rangle_{X\times X^*}
= & [i_\infty^*f,i_\infty^*g]_{H_\infty}
=\langle R^*f,R^*g\rangle_{\elle^2(X,\mu_\infty)}=\langle Q_\infty f,g\rangle_{X\times X^*},
\end{align}
which gives $Q_\infty=i_\infty\circ i_\infty^*$.

We introduce the following spaces of functions, which have been already considered in \cite{MV07,MV08}.
\begin{defn}
For any $k\in\N\cup\{\infty\}$ we set
\begin{align*}
\fcon_{b}^{k,1}(X)
:=\{ & f(x)=\varphi(\langle x,x_1^*\rangle_{X\times X^*},\ldots,\langle x,x_n^*\rangle_{X\times X^*}):\ n\in\N, \ \varphi\in C_b^k(\R^n), \ x_i\in D(A^*) ,\\
&  i=1,\ldots,n, \ x\in X\}.
\end{align*}
\end{defn}

%\begin{lemma}
%\label{ortonormal_basis}
%$H_\infty$ admits an orthonormal basis $\Theta:=\{e_n:n\in\N\}$ such that $e_n=i^*_\infty x^*_n$ with $x^*_n\in D(A^*)$ for any $n\in\N$. 
%\end{lemma}
%\begin{proof}
%It is well known (see e.g. \cite[Theorem 2.2]{K65}) that the weak$^*$-closure of $D(A^*)$ coincides with $X^*$. Then, for any $x^*\in X^*$ there exists a sequence $(x^*_n)\subset D(A^*)$ such that $x^*_n\rightarrow x^*$ in the weak$^*$-topology, that is, $\langle x,x^*_n\rangle_{X\times X^*}\rightarrow \langle x,x^*\rangle_{X\times X^*}$ for any $x\in X$. Therefore, for any $x\in X$ there exists a positive constant $c_x$ such that $\sup_{n\in\N}|\langle x,x^*_n\rangle_{X\times X^*}|\leq c_x$. The uniform boundedness principle gives $\sup_{n\in\N}\|x^*_n\|_{X^*}\leq c$ for some positive constant $c$. By the dominated convergence theorem and the Fernique Theorem it follows that $R^*x^*_n\rightarrow R^*x^*$ in $\elle^2(X,\mu_\infty)$. Combining this fact and \eqref{prop_covariance_op} gives
%\begin{align*}
%|i^*_\infty x^*_n-i^*_\infty x^*|_{H_\infty}^2
%= & \int_X|\langle x,x^*_n-x^*\rangle_{X\times X^*}|^2\mu_\infty(dx)\rightarrow 0,
%\end{align*}
%as $n\rightarrow +\infty$. Therefore, $Q_\infty (D(A^*))$ is dense in $Q_\infty X$ with respect to  $|\cdot |_{H_\infty}$. Since from \cite[Corollary 3.2.8]{Bog98} $Q_\infty X$ is dense in $H_\infty$, we conclude that $Q_\infty (D(A^*))$ is dense in $H_\infty$. In particular, this implies that there exists an orthonormal basis of $H_\infty$ of elements of $Q_\infty (D(A^*))$. 
%\end{proof}

\begin{rmk}
\label{eq_cyl_functions}
We stress that the spaces $\fcon_{b}^{k,1}(X)$ are different from those considered in \cite{ACF19,CF16,DPL14,GGV03}. Indeed, in these papers the authors consider the spaces $\fcon_{b}^{k}(X)$, %or $\fcon_b^{k,\ell}(X)$, with $k,\ell\in\N$, 
that is, the spaces of cylindrical functions $f$ such that $f(x)=\varphi(\langle x,y_1^*\rangle_{X\times X^*}, \ldots,\langle x,y_n^*\rangle_{X\times X^*})$ for any $x\in X$, for some $\varphi\in C_b^k(\R^n)$ and $y^*_1,\ldots,y^*_n\in X^*$.
Even if the space $\fcon_{b}^{k,1}(X)$ is smaller than $\fcon_{b}^{k}(X)$ % and of $\fcon_b^{k,1}(X)$, 
it is "good" in the sense that it is big enough. %, since $\{x^*_n:n\in\N\}$ is an orthonormal basis of $H_\infty$. Further,
Indeed, from \cite[Theorem 2.2]{K65} it follows that $D(A^*)$ is weak$^*$-dense in $X^*$. Since $\fcon_{b}^{k}(X)$ is dense in $\elle^p(X,\nu_\infty)$ for any $p\in[1,+\infty)$ and any $k\in\N$ (see \cite[Corollary 3.5.2]{Bog98}), we get that $\fcon_{b}^{k,1}(X)$ is dense in $\elle^p(X,\nu_\infty)$ for any $p\in[1,+\infty)$ and any $k\in\N$.
\end{rmk}

\begin{example}
\label{example1}
We provide a construction of the classical Wiener space by means of special operators $A$ and $Q$. We consider the classical Wiener space $(X,H_\infty,\nu_\infty)$, where $X=\elle^2(0,1)$, $H_\infty=\{f\in W^{1,2}(0,1):f(0)=0\}$ and $\mu_\infty=P^W$ is the classical Wiener measure, see e.g. \cite[Example 2.3.11 \& Remark 2.3.13]{Bog98}. Let us denote by $Q_\infty$ its covariance operator and $Q:=Q_\infty^{1/2}$. Then, if we set
\begin{align*}
D(A):=QX, \quad A:=-Q,
\end{align*}
$(A,D(A))$ is a closed operator with dense domain satisfying $\langle A f,f\rangle_{\elle^2(X,\mu_\infty)}\leq 0$ for any $f\in D(A)$. Therefore, $A$ generates an analytic semigroup which is also strongly continuous. Further, we have
\begin{align*}
Q_t=Q_\infty(Id_X-e^{tA}), \quad t>0,
\end{align*}
which implies that $Q_t$ is a trace class operator for any $t>0$ and the covariance operator $Q_\infty$ coincides with the integral
\begin{align*}
\int_0^{+\infty}e^{tA}Qe^{tA}dt.
\end{align*}
\end{example}

\subsection{Reproducing Kernel associated to \texorpdfstring{$Q$}{Q} and Sobolev Spaces}
%Starting from \eqref{inner_prod_inf_dual} we can define the Reproducing Kernel Hilbert Space associated to $Q$ (see also \cite{VN98}). 
%The relevance is that along the directions of this Reproducing Kernel it will be possible to define a gradient which we will show being a closable operator.
We recall that $Q$ is a bounded, linear, nonnegative and symmetric operator. From Definition \ref{RKHS} we can define a scalar product on $QX^*$ and we denote by $H$ the Reproducing Kernel Hilbert Space associated to $Q$. $H$ is a Hilbert space if endowed with the scalar product $[\cdot,\cdot]_H$. The inclusion $QX^*\hookrightarrow X$ can be extended to the injection $i:H\rightarrow X$ and we consider the adjoint operator $i^*:X^*\rightarrow H$, where again we have identify $H^*$ and $H$. Arguing as for $i_\infty$ and $i_\infty^*$ we infer that $Q=i\circ i^*$.

%From the definition, $Q_\infty$ is a positive symmetric operator, arguing as for $Q$ we can define the RKHS $H_\infty$ associated to $Q_\infty$, the injection $i_\infty:H_\infty\rightarrow X$ and its adjoint operator $i_\infty^*:X^*\rightarrow H_	\infty$. 

The following hypothesis is very important since \cite[Theorem 8.3]{GV03} states that it is equivalent to the analyticity in $\elle^p(X,\mu_\infty)$ of the Ornstein-Uhlenbeck semigroup $P(t)$ defined on $C_b(X)$ by
\begin{align*}
(P(t)f)(x):=\int_Xf(e^{tA}x+y)\mu_t(dy), \quad f\in C_b(X),
\end{align*}
and extended to $\elle^p(X,\mu_\infty)$ for any $p\in(1,+\infty)$. 

\begin{hyp}
\label{ipo_RKH}
For any $x^*\in D(A^*)$ we have $i^*_\infty A^*x^*\in H$ and there exists a positive constant $c$ such that
\begin{align}
\label{schema}
|i^*_\infty A^* x^*|_H\leq c|i^*x^*|_H, \qquad x\in D(A^*).
\end{align}
\end{hyp}

$i^*$ is continuous with respect to the weak$^*$ topology on $X^*$ and to the weak topology on $H$. Since $D(A^*)$ is weak$^*$-dense in $X^*$, it follows that $i^*$ maps $D(A^*)$ onto a dense subspace of $H$. Then, there exists an operator $B\in\mathcal L(H)$ such that $Bi^*x^*=i^*_\infty A^*x^*$ for any $x^*\in D(A^*)$ and $\|B\|_{\mathcal L(H)}\leq c$. The operator $B$ enjoys the following properties.

\begin{lemma}
\label{propr_B}
{\cite[Lemma 2.2]{MV07}}
$B+B^*=-Id_H$ and $[Bh,h]_H=-\frac12|h|^2_H$ for any $h\in H$.
\end{lemma}

%From \cite[Fernique Theorem 2.8.5]{Bog98} the bounded linear functionals, seen as functions from $X$ to $\R$, belong to $\elle^p(X,\mu_\infty)$ for any $p\in[1,+\infty)$. We introduce the functional $\Phi$ which has been already used in \cite{GGV03,MV07}. 

%\begin{defn}
%We set $\Phi:i^*_\infty X^*\rightarrow \elle^2(X,\mu_\infty)$ defined by $i^*_\infty x^*\mapsto \left(x\mapsto \langle \cdot,x^*\rangle_{X\times X^*}\right)$ for any $x^*\in X^*$.
%\end{defn}

We now introduce two operators which are crucial for the definition of Sobolev spaces in our context. The first one is the gradient along the directions of the Reproducing Kernel Hilbert Space $H$, while the second one allows us to prove an integration by parts formula with respect to suitable directions in $H$ (see e.g. \cite[Section 3]{GGV03}).
\begin{defn}
We define the operator $D_H:\fcon_{b}^{1}(X)\rightarrow \elle^p(X,\mu_\infty;H)$ by
\begin{align*}
D_Hf(x):=i^*Df(x)=\sum_{j=1}^n\frac{\partial \varphi}{\partial \xi_j}(\langle x, x^*_1\rangle_{X\times X^*},\ldots,\langle x,x^*_n\rangle_{X\times X^*}) i^*x^*_j, \quad x\in X,
\end{align*}
where $f\in \fcon_{b}^{1}(X)$ and $f(x)=\varphi(\langle x,x^*_1\rangle_{X\times X^*},\ldots,\langle x, x^*_n\rangle_{X\times X^*})$ for some $n\in\N$, $\varphi\in C^1_b(\R^n)$, $x_i^*\in X^*$ for $i=1,\ldots,n$ and for any $x\in X$.
\end{defn}

\begin{defn}
We define the operator $V:D(V)\subseteq H_\infty\rightarrow H$ as follows:
\begin{align}
\label{operator_V}
D(V):=\{i^*_\infty x^*:x^*\in X^*\}, \quad V(i^*_\infty x^*)=i^*x^*, \quad x^*\in X^*.
\end{align}
\end{defn}
$V$ is densely defined on $H_\infty$, then it is possible to consider the adjoint operator $V^*:D(V^*)\subset H\rightarrow H_\infty$. Thanks to Hypothesis \ref{ipo_RKH} and \cite[Theorems 8.1, 8.3 \& Proposition 8.7]{GV03} it follows that $D_H$ is closable in $\elle^p(X,\mu_\infty)$ and \cite[Theorem 3.5]{GGV03} gives that the operator $V$ is closable. We still denote by $D_H$ the closure of $D_H$ and by $W^{1,p}_H(X,\mu_\infty)$ the domain of the closure. We set
\begin{align*}
\|f\|_{W^{1,p}_H(X,\mu_\infty)}:=\|f\|_{\elle^p(X,\mu_\infty)}+\|D_Hf\|_{\elle^p(X,\mu_\infty;H)}, \quad f\in W^{1,p}_H(X,\mu_\infty).
\end{align*}

The following lemma shows that $\fcon_b^{1,1}(X)$ is dense in $W^{1,p}_H(X,\mu_\infty)$ for any $p\in(1,+\infty)$.
\begin{lemma}
\label{convergenza_fcon}
Let $f\in \fcon_b^1(X)$. Then, for any $p\in(1,+\infty)$ there exists a sequence $(f_n)\subset \fcon_{b}^{1,1}(X)$ such that $f_n\rightarrow f$ in $W^{1,p}_H(X,\mu_\infty)$ as $n\rightarrow+\infty$. In particular, this gives that $\fcon_b^{1,1}(X)$ is dense in $W^{1,p}_H(X,\mu_\infty)$ for any $p\in(1,+\infty)$.
\end{lemma}
\begin{proof}
We recall that $D(A^*)$ is weak$^*$-dense in $X^*$ (see \cite[Theorem 2.2]{K65}). This implies that for any $x^*\in X^*$ there exists a sequence $(x_m^*)\subset D(A^*)$ which weak$^*$ converges to $x^*$ as $m\rightarrow+\infty$, i.e., $\langle x,x^*_m\rangle_{X\times X^*}\rightarrow \langle x,x^*\rangle_{X\times X^*}$ as $m\rightarrow+\infty$ for any $x\in X$. 

We claim that for any $h\in H$ and any $x^*\in X^*$ we have $[i^*x^*,h]_H=\langle h,x^*\rangle_{X\times X^*}$. From the definition of $H$, this is true when $h=i^*y^*$ for some $y^*\in X^*$. For a generic $h\in H$, let $(x^*_n)\subset X^*$ be such that $i^*x^*_n\rightarrow h$ in $H$ as $n\rightarrow +\infty$. Since $H\subset X$ with continuous embedding, it follows that $(i\circ i^*)x^*_n\rightarrow h$ in $X$ as $n\rightarrow+\infty$. Then,
\begin{align}
\label{catena_ug_H}
[i^*x^*,h]_H
= & \lim_{n\rightarrow+\infty}[i^*x^*,i^*x^*_n]_H
= \lim_{n\rightarrow+\infty}\langle (i\circ i^*)x_n^*,x^*\rangle_{X\times X^*}
=\langle h,x^*\rangle_{X\times X^*},
\end{align}
and the claim is so proved.

Let $f\in \fcon_b^1(X)$. We only consider $f(x)=\varphi(\langle x,x^*\rangle_{X\times X^*})$ with $\varphi\in C^1_b(\R)$, $x^*\in X^*$ and $x\in X$, the general case easily follows from this one. We set $\tilde f_n:=\varphi(\langle x, x^*_n\rangle_{X\times X^*})$, where $(x_n^*)\subset D(A^*)$ is a sequence which weak$^*$ converges to $x^*$ as $n\rightarrow+\infty$. Then, $\tilde f_n(x)\rightarrow f(x)$ pointwise, and the dominated convergence theorem gives that $\tilde f_n\rightarrow f$ in $\elle^p(X,\mu_\infty)$ as $n\rightarrow+\infty$ for any $p\in[1,+\infty)$. 

Let us fix $p\in(1,+\infty)$. We show that there exists a sequence $(f_n)\subset \fcon_b^{1,1}(X)$ such that $f_n\rightarrow f$ in $\elle^p(X,\mu_\infty)$ and $D_Hf_n\rightarrow D_Hf$ in $\elle^p(X,\mu_\infty;H)$ as $n\rightarrow+\infty$. From the definition of $D_H$ we have
\begin{align*}
D_H\tilde f_n(x)=\varphi'(\langle x, x^*_n\rangle_{X\times X^*})i^*x^*_n, \quad x\in X, \ n\in\N.
\end{align*}
 From \eqref{catena_ug_H}, for any $h\in H$ we get
\begin{align*}
[i^*x_n^*,h]_H
=\langle h,x_n^*\rangle_{X\times X^*}\rightarrow \langle h, x^*\rangle_{X\times X^*}=[i^*x^*,h]_H, \quad n\rightarrow+\infty.
\end{align*}
This implies that $(i^*x^*_n)\subset H$ weakly converges in $H$ to $i^*x^*$ as $n\rightarrow+\infty$ and so the sequence $(i^*x^*_n)$ is bounded in $H$. Therefore, there exists a positive constant $c_p$ such that $\|\tilde f_n\|_{W^{1,p}_H(X,\mu_\infty)}\leq c_p$ for any $n\in\N$. From \cite[Chapter 3]{D75} we deduce that $\elle^p(X,\nu_\infty;H)$ is uniformly convex for any $p\in(1,+\infty)$, and so $\elle^p(X,\nu_\infty;H)$ has the Banach-Saks property (see e.g. \cite[Theorem 1, pag. 78]{D75}). We apply this property to the bounded sequence $(D_H\tilde f_n)$, hence there exists a subsequence $(D_H\tilde f_{k_n})\subset (D_H\tilde f_n)$ such that if we set 
\begin{align*}
f_n:=\sum_{i=1}^{n}\frac{\tilde f_{k_1}+\ldots+\tilde f_{k_n}}{n}, \quad n\in\N,
\end{align*}
the sequence
\begin{align*}
D_Hf_n:=\sum_{i=1}^{n}\frac{D_H\tilde f_{k_1}+\ldots+D_H \tilde f_{k_n}}{n}, \quad n\in\N,
\end{align*}
converges to a function $\Psi$ in $\elle^p(X,\mu_\infty;H)$ as $n\rightarrow+\infty$. Clearly, $f_n\rightarrow f$ as $n\rightarrow+\infty$ in $\elle^p(X,\mu_\infty)$. From the fact that $D_H$ is a closed operator on $\elle^p(X,\nu_\infty)$, we infer that $\Psi=D_Hf$. To conclude, we notice that $f_n\in \fcon_b^{1,1}(X)$ for any $n\in\N$.
\end{proof}

\begin{lemma}
\label{domV^*}
For any $x^*\in D(A^*)$, we have $Bi^*x^*\in D(V^*)$ and $V^*(Bi^*x^*)=i^*_\infty A^*x^*$.
\end{lemma} 
\begin{proof}
The result is contained in the proof of \cite[Theorem 2.3]{MV07}, but for reader's convenience we provide the simple proof. Let $x^*\in D(A^*)$. From the definition of $[\cdot,\cdot]_H$, that of $[\cdot,\cdot]_{H_\infty}$ and that of $V$, for any $y^*\in X^*$  we have
\begin{align*}
[Bi^*x^*,V(i^*_\infty y^*)]_H
= & [Bi^*x^*,i^*y^*]_H
= [i^*_\infty A^*x^*,i^*y^*]_H
= \langle i^*_\infty A^*x^*,y^*\rangle_{X\times X^*}
=[i^*_\infty A^*x^*,i^*_\infty y^*]_{H_\infty},
\end{align*}
which means that $Bi^*x^*\in D(V^*)$ and $V^*(Bi^*x^*)=i^*_\infty A^* x^*$.
\end{proof}

\begin{rmk}
If $Q=Q_\infty$, i.e., the Malliavin setting, $D_H$ is the Malliavin derivative, $V$ is the identity operator and for any $p\in[1,+\infty)$ the space $W^{1,p}_H(X,\mu_\infty)$ is the Sobolev space considered in \cite[Chapter 5]{Bog98}.
\end{rmk}

%\begin{rmk}
%Since $(X,\mu_\infty,H_\infty)$ is a Wiener space, we can always consider the  Malliavin derivative $D_{H_\infty}$ and the  Sobolev spaces $W^{1,p}(X,\mu_\infty)$ (see e.g. \cite[Chapter 5]{Bog98}).
%\end{rmk}

%\textcolor{red}
%{\begin{example}
%\label{example2}
%In the setting of Example \ref{example1}, the space $W^{1,p}_H(X,\mu_\infty)$ is a generalization of the  Sobolev spaces. Indeed, it is well known that if $X$ is a Hilbert space, then the Cameron-Martin space $H_\infty$ coincides with $Q_\infty^{1/2}$. With the notations introduced in \ref{example1}, the Cameron-Martin space $H$ is $Q^{1/2}=Q_\infty^{1/4}$. This means that when we differentiate along the directions of $H$ we are considering more directions with respect the derivatives along $H_\infty$. In particular, this means that $W^{1,p}(X,\mu_\infty)\subset W^{1,p}_H(X,\mu_\infty)$.
%\end{example}}

%\begin{rmk}
%It is not hard to see that, even if we consider a space of test functions which is smaller with respect to those considered in \cite{MV07,MV08}, we obtain the same Sobolev space $W^{1,p}_H(X,\mu_\infty)$ for any $p\in[1,+\infty)$.
%\end{rmk}

We are now ready to state the hypotheses on the weighted function $U$.

\begin{hyp}
\label{ventola}
$U$ is a proper $\|\cdot\|_X$-lower semi-continuous convex function which belongs to $W_H^{1,p}(X,\mu_\infty)$ for any $p\in[1,+\infty)$.
\end{hyp}

It is useful to notice that Hypothesis \ref{ventola} and \cite[Lemma 7.5]{AB06} imply that $e^{-U}\in W_H^{1,p}(X,\mu_\infty)$ for any $p\in[1,+\infty)$. This allows us to introduce the bounded measure 
\begin{align}
\label{weighted_measure}
\nu_\infty:=e^{-U}d\mu_\infty.
\end{align}

We prove that $D_H:\fcon_{b}^{1}(X)\rightarrow \elle^p(X,\nu_\infty;H)$ is closable in $\elle^p(X,\nu_\infty)$. To this aim we prove an intermediate result, which is the extension of \cite[Lemma 3.3]{GGV03} for the weighted measure $\nu_\infty$.

\begin{lemma}
Let $f\in \fcon_{b}^{1}(X)$ and let $h\in D(V^*)$. Then,
\begin{align}
\label{int_parti_peso}
\int_X[D_Hf,h]_H d\nu_\infty
=\int_X f\widehat{V^* h}d\nu_\infty+\int_Xf[D_HU,h]_Hd\nu_\infty.
\end{align}
\end{lemma}

\begin{proof}
From \cite[Lemma 3.3]{GGV03} we know that 
\begin{align}
\label{ibpf1}
\int_X[D_Hg,h]_H d\mu_\infty
=\int_X g\widehat{V^* h}d\mu_\infty,
\end{align}
for any $g\in \fcon_{b}^{1}(X)$ and any $h\in D(V^*)$. We would like to apply \eqref{ibpf1} with $g=fe^{-U}$. The density of $\fcon_{b}^{1}(X)$ in $W_H^{1,p}(X,\mu_\infty)$ for any $p\in[1,+\infty)$ implies that \eqref{ibpf1} holds true for any $g\in W_H^{1,p}(X,\mu_\infty)$ and $p\in[1,+\infty)$. From Hypothesis \ref{ventola} and \cite[Lemma 3.3]{MV08}, we infer that $D_H (fe^{-U})=(D_H f)e^{-U}-(D_HU)fe^{-U}$. Then, $f e^{-U}\in W_H^{1,p}(X,\mu_\infty)$ for any $p\in[1,+\infty)$ and we can apply \eqref{ibpf1} with $g=fe^{-U}$. We get
\begin{align*}
\int_X [D_H f,h]_Hd\nu_\infty= &
\int_X [D_H f,h]_He^{-U}d\mu_\infty
= \int_X [D_H(fe^{-U}),h]_Hd\mu_\infty+\int_X f[D_HU,h]_He^{-U}d\mu_\infty \\
 =&  \int_X fe^{-U}\widehat{V^* h}d\mu_\infty+\int_Xf[D_HU,h]_Hd\nu_\infty \\
=&  \int_X f\widehat{V^* h}d\nu_\infty+\int_Xf[D_HU,h]_Hd\nu_\infty.
\end{align*}
\end{proof}

Integration by parts formula \eqref{int_parti_peso} is the key tool to prove the closability of $D_H$ in $\elle^p(X,\nu_\infty)$ with $p\in(1,+\infty)$.

\begin{pro}
\label{clos_D_H_peso}
$D_H:\fcon_{b}^{1}(X)\rightarrow \elle^p(X,\nu_\infty;H)$ is closable in $\elle^p(X,\nu_\infty)$ for any $p\in(1,+\infty)$. We still denote by $D_H$ the closure of $D_H$ in $\elle^p(X,\nu_\infty)$ and we denote by $W^{1,p}_H(X,\nu_\infty)$ the domain of its closure. Finally, for any $p\in(1,+\infty)$ the space $W^{1,p}_H(X,\nu_\infty)$ endowed with the norm
\begin{align*}
\|f\|_{1,p,H}:=\|f\|_{\elle^p(X,\nu_\infty)}+\|D_Hf\|_{\elle^p(X,\nu_\infty;H)}, \quad f\in W^{1,p}_H(X,\nu_\infty),
\end{align*}
is a Banach space, and for $p=2$ it is a Hilbert space with inner product
\begin{align*}
\langle f,g\rangle_{W^{1,2}_H(X,\nu_\infty)}
:=\int_Xfgd\nu_\infty+\int_X[D_Hf,D_Hg]_Hd\nu_\infty, \quad f,g\in W^{1,2}_H(X,\nu_\infty).
\end{align*}
\end{pro}

\begin{proof}
Let us fix $p\in(1,+\infty)$. $(V,D(V))$ is closable from $H_\infty$ onto $H$, then from \cite[Theorem 3.4]{GGV03} it follows that $D(V^*)$ is weak dense in $H$ and there exists an orthonormal basis $\{v_n:n\in\N\}\subset D(V^*)$ of $H$. To show that $D_H$ is closable, let us consider a sequence $(f_n)\subset \fcon_{b}^{1}(X)$ such that $f_n\rightarrow 0$ and $D_Hf_n\rightarrow F$ in $\elle^p(X,\nu_\infty)$ and in $\elle^p(X,\nu_\infty;H)$, respectively. If we show that $F=0$ we infer the closability of $D_H$. To prove that $F=0$ let us consider $g\in \fcon_{b}^{1}(X)$. From \eqref{int_parti_peso} applied to the function $\tilde f_n:=f_n g\in \fcon_{b}^{1}(X)$ we have
\begin{align}
\int_X[D_Hf_n,v_j]_Hg d\nu_\infty 
=& \int_X[D_H(f_n g),v_j]_H d\nu_\infty-\int_X[D_Hg,v_j]_H f_nd\nu_\infty \notag \\
= & \int_Xf_ng\widehat{V^*v_j}d\nu_\infty+\int_X[D_HU,v_j]_H f_n g d\nu_\infty-\int_X[D_Hg,v_j]_H f_nd\nu_\infty, \label{int_parti_completa}
\end{align}
for any $j\in\N$. Letting $n\rightarrow+\infty$ in \eqref{int_parti_completa} we infer that
\begin{align*}
\int_X[F,v_j]_Hg d\nu_\infty 
= \lim_{n\rightarrow+\infty}\int_X[D_Hf_n,v_j]_Hg d\nu_\infty =0,
\end{align*}
for any $j\in\N$ and any $g\in\fcon_{b}^{1}(X)$. The density of $\fcon_{b}^1(X)$ in $\elle^p(X,\nu_\infty)$ implies that $[F(x),v_j]_H=0$ for $\nu_\infty$-a.e. $x\in X$ for any $j\in\N$. This gives that $F(x)=0$ for $\nu_\infty$-a.e. $x\in X$.
The second part of the statement follows from standard arguments.
\end{proof}

\begin{rmk}
\label{eq_chiusura_dom_C2}
Arguing as in Lemma \ref{convergenza_fcon}, it follows that the space $\fcon_b^{k,1}(X)$ is dense in $W^{1,p}_H(X,\nu_\infty)$ for any $k\in\N\cup\{\infty\}$ and any $p\in(1,+\infty)$.
\end{rmk}

\section{The perturbed nonsymmetric Ornstein-Uhlenbeck operator}
\label{nonsymm_OU_op_section}
 \subsection{The Ornstein-Uhlenbeck operator in \texorpdfstring{$\elle^2(X,\nu_\infty)$}{}}
\label{nonsymm_OU_op}

We introduce the nonsymmetric Ornstein-Uhlenbeck operator by means of the theory of bilinear Dirichlet forms. Let
\begin{align}
\label{rough}
\mathcal E(u,v):=-\int_X [BD_Hu,  D_Hv]_H d\nu_\infty,  \quad u,v\in\mathcal D,
\end{align}
with domain $\mathcal D=W^{1,2}_H(X,\nu_\infty)$. From Lemma \ref{propr_B} we get
\begin{align}
\mathcal E(u,u)
= & -\int_X[ BD_Hu,D_Hu]_Hd\nu_\infty
=\frac 12 \int_X[D_Hu,D_Hu]_H d\nu_\infty =\frac12 \|D_Hu\|_{\elle^2(X,\nu_\infty;H)}^2, \quad u\in\mathcal D,
\label{topolino}
\end{align}
which implies that $\mathcal E$ is positive definite. If we consider the symmetric part $\overline{ \mathcal E}(u,v):=\frac12(\mathcal E(u,v)+\mathcal E(v,u))$ of $\mathcal E$, with $u,v\in \mathcal D$, we have
\begin{align*}
\overline {\mathcal E}(u,v)
=&- \frac12\int_X([BD_Hu,D_Hv]_H+[BD_Hv,D_Hu]_H)d\nu_\infty \\
= &- \frac12\int_X([BD_Hu,D_Hv]_H+[B^*D_Hu,D_Hv]_H)d\nu_\infty=\frac12\int_X[D_Hu,D_Hv]d\nu_\infty.
\end{align*}
Proposition \ref{clos_D_H_peso} implies that $(\overline{\mathcal E},\mathcal D)$ is a symmetric closed form on $\elle^2(X,\nu_\infty)$. Finally, for any $u,v\in \mathcal D$, from Hypothesis \ref{ipo_RKH} we have
\begin{align*}
|\mathcal E(u,v)|
\leq & \int_X|[BD_Hu,D_Hv]_H|d\nu_\infty
\leq  \|B\|_{\mathcal L(H)}\int_X|D_Hu|_H |D_Hv|_Hd\nu_\infty \\
\leq & c\ \!\|D_Hu\|_{\elle^2(X,\nu_\infty;H)} \|D_Hv\|_{\elle^2(X,\nu_\infty;H)}
=4c\ \!\mathcal E(u,u)^{1/2}\mathcal E(v,v)^{1/2}.
\end{align*}
This implies that $({\mathcal E},\mathcal D)$ satisfies the {\it strong} (and hence the {\it weak}) {\it sector condition} (see \cite[Chapter 1, Section 2 and Exercise 2.1]{MR92}) and therefore $({\mathcal E},\mathcal D)$ is a coercive closed form on $\elle^2(X,\nu_\infty)$.
According to \cite[Chapter 1]{MR92} we define a densely defined operator $L_2$ as follows:
\begin{align}
\label{cerniera}
\left\{
\begin{array}{ll}
D(L_2) :=\Big\{u\in W^{1,2}_H(X,\nu_\infty):\ {\textrm{there exists $g\in \elle^2(X,\nu_\infty)$ such that }}\\
\qquad \qquad \quad \qquad \qquad \qquad \qquad \displaystyle \mathcal E(u,v)=-\int_X gvd\nu_\infty, \ \forall v\fcon_{b}^{1}(X)\Big\}, \\
L_2u :=g.
\end{array}
\right.
\end{align}

\begin{rmk}
\label{mirto}
From \cite[Chapter 1, Sections 1\&2]{MR92} it follows that $L_2$ generates a strongly continuous contraction semigroup on $\elle^2(X,\nu_\infty)$ which we denote by $(T_2(t))_{t\geq0}$. In particular, $1\in \rho(L_2)$. The operator $L_2$ is called {\textit{perturbed Ornstein-Uhlenbeck operator in $\elle^2(X,\nu_\infty)$}} and the associated semigroup $(T_2(t))_{t\geq0}$ is called {\textit{perturbed Ornstein-Uhlenbeck semigroup in $\elle^2(X,\nu_\infty)$}}.
\end{rmk}

In the following we will need of the adjoint operator $L_2^*$ of $L_2$. We recall that formally $L_2^*$ is defined as follows:
\begin{align*}
\left\{
\begin{array}{l}
\displaystyle D(L_2^*):= \Big\{v\in \elle^2(X,\nu_\infty):\exists g\in \elle^2(X,\nu_\infty) \textrm{ such that}\vspace{1mm} \\
\hphantom{D(L_2^*):=\big\{v\in \elle^2(X,\nu_\infty):}
\displaystyle \int_X gud\nu_\infty=\int_XvL_2ud\nu_\infty, \quad u\in D(L_2)\Big\}, \\
L_2^*v :=g.
\end{array}
\right.
\end{align*}
Moreover, let us consider the adjoint semigroup $(T_2^*(t))_{t\geq0}$ of $(T_2(t))_{t\geq0}$. Even if in general it is not a strongly continuous semigroup, \cite[Chapter 1, Theorem 2.8]{MR92} ensures that $(T_2^*(t))_{t\geq0}$ is strongly continuous and $L_2^*$ is its generator. Further, \cite[Chapter 1, Corollary 2.10]{MR92} implies that $D(L_2^*)\subset \mathcal D=W^{1,2}_H(X,\nu_\infty)$.

We give a characterization of $L_2^*$ in terms of bilinear form on $\elle^2(X,\nu_\infty)$. Let us introduce the nonsymmetric bilinear form
\begin{align}
\label{nonsym_bil_form_dual}
\widetilde {\mathcal E}(u,v):=-\int_X[B^*D_Hu,D_Hv]_Hd\nu_\infty, \quad u,v\in\mathcal D,
\end{align}
with domain $\mathcal D:=W^{1,2}_H(X,\nu_\infty)$. Arguing as for $\mathcal E$ it is possible to prove that $\widetilde {\mathcal E}$ is a coercive closed form on $\elle^2(X,\nu_\infty)$ and therefore the operator $\widetilde L_2$ defined as
\begin{align}
\label{del_op_adj}
\left\{
\begin{array}{ll}
D(\widetilde L_2) :=\Big\{u\in W^{1,2}_H(X,\nu_\infty):\ {\textrm{there exists $g\in \elle^2(X,\nu_\infty)$ such that }}\\
\qquad \qquad \quad \qquad \qquad \qquad \qquad \displaystyle \widetilde{\mathcal E}(u,v)=-\int_X gvd\nu_\infty, \ \forall v\fcon_{b}^{1}(X)\Big\}, \\
\widetilde L_2u :=g,
\end{array}
\right.
\end{align}
generates a strongly continuous semigroup $(\widetilde T_2(t))_{t\geq0}$ on $\elle^2(X,\nu_\infty)$. The next result shows that $\widetilde L_2$ is indeed the adjoint operator of $L_2$ and $(\widetilde T_2(t))_{t\geq0}$ is the adjoint semigroup of $(T_2(t))_{t\geq0}$.
\begin{pro}
$D(\widetilde L_2)=D(L_2^*)$ and $\widetilde L_2u=L_2^*u$ for any $u\in D(L_2^*)$. Therefore, $\widetilde T_2(t)=T_2^*(t)$ for any $t\geq0$.
\end{pro}
\begin{proof}
Let $u\in D(\widetilde L_2)$. For any $v\in D(L_2)$ we have
\begin{align*}
\int_X\widetilde L_2 u vd\nu_\infty
= & \int_X[B^*D_Hu,D_Hv]_Hd\nu_\infty
= \int_X[BD_Hv,D_Hu]d\nu_\infty
= \int_XL_2vud\nu_\infty.
\end{align*}
From the definition of $L_2^*$ it follows that $u\in D(L_2^*)$ and $L_2^*u=\widetilde L_2u$. To prove the converse inclusion, let $u\in D(L_2^*)$. We recall that $u\in W^{1,2}_H(X,\nu_\infty)$. For any $v\in D(L_2)$ we have
\begin{align}
\label{DLtilde_cont_DLadj}
\int_XL_2^* uvd\nu_\infty
= & \int_XuL_2vd\nu_\infty
=\int_X[BD_Hv,D_Hu]_Hd\nu_\infty
= \int_X[B^*D_Hu,D_Hv]_Hd\nu_\infty=-\widetilde {\mathcal E}(u,v).
\end{align}
From \cite[Chapter 1, Theorem 2.13(ii)]{MR92} it follows that $D(L_2)$ is dense in $\mathcal D=W^{1,2}_H(X,\nu_\infty)$. Therefore, \eqref{DLtilde_cont_DLadj} gives $u\in D(\widetilde L_2)$ and $\widetilde L_2u=L_2^*u$.
\end{proof}

We conclude this subsection by showing that $\fcon_{b}^{2,1}(X)\subset D(L_2)$ and for any $u\in \fcon_{b}^{2,1}(X)$ an explicit formula for $L_2u$ is available. To this aim, we recall the definition of Trace class operator on $\mathcal L(H)$: given a nonnegative operator $\Phi\in \mathcal L(H)$, we say that $\Phi$ is a trace class operator if 
\begin{align*}
\sum_{n=1}^\infty[\Phi h_n,h_n]_H<+\infty,
\end{align*}
where $\{h_n:n\in\N\}$ is any orthonormal basis of $H$. We define the Trace ${\rm Tr}[\Phi]$ of $\Phi$ as
\begin{align*}
{\rm Tr}[\Phi]_H:=\sum_{n=1}^\infty[\Phi h_n,h_n]_H.
\end{align*}

For any $f\in \fcon_{b}^{2,1}(X)$ such that $f(x)=\varphi(\langle x, x^*_1\rangle_{X\times X^*},\ldots,\langle x,x^*_n\rangle_{X\times X^*})$ for some $\varphi\in C_b^2(\R^n)$, $x_i^*\in D(A^*)$, $i=1,\ldots,n$ and $x\in X$, we define the second order derivative along $H$ as
\begin{align*}
D^2_Hf(x):=\sum_{j,k=1}^n\frac{\partial^2\varphi}{\partial \xi_j\xi_k}(\langle x, x^*_1\rangle_{X\times X^*},\ldots,\langle x,x^*_n\rangle_{X\times X^*})Qx^*_j\otimes Qx_k^*.
\end{align*}
$D^2_Hf(x)$ is a trace class operator for any $x\in X$ and
\begin{align*}
{\rm Tr}[D^2_Hf(x)]_H
=\sum_{j,k=1}^n\langle Q x^*_j,x^*_k\rangle_{X\times X^*}\frac{\partial^2 \varphi}{\partial\xi_j\partial \xi_k}(\langle x, x^*_1\rangle_{X\times X^*},\ldots,\langle x,x^*_n\rangle_{X\times X^*}), \quad x\in X.
\end{align*}

\begin{pro}
\label{gruppo}
$\fcon_{b}^{2,1}(X)\subset D(L_2)$ and for any $u\in\fcon_{b}^{2,1}(X)$ we have
\begin{align}
\label{pingpong}
L_2u(x)=\frac12{\rm Tr}[D_H^2u(x)]_H+\langle x,A^*Du(x)\rangle_{X\times X^*} +[BD_Hu(x),D_HU(x)]_H, \quad \nu_\infty{\rm -a.e.}\ x\in X.
\end{align}
\end{pro}
\begin{proof}
Let $u\in \fcon_{b}^{2,1}(X)$ be such that $u(x)=\varphi(\langle x, x^*_1\rangle_{X\times X^*},\ldots,\langle x,x^*_m\rangle_{X\times X^*})$, with $\varphi\in C_b^2(\R^m)$, $x_i^*\in D(A^*)$ for $i=1,\ldots,m$ and $x\in X$, and let $v\in \fcon_{b}^{1}(X)$. From Lemma \ref{domV^*} for any $x^*\in D(A^*)$ we have $Bi^*x^*\in D(V^*)$ and $V^*(Bi^*x^*)=i^*_\infty A^*x^*$. The form of $u$, integration by parts formula \eqref{int_parti_peso} with $f=uv$ and the computations in the proof of \cite[Theorem 2.3]{MV07} give
\begin{align*}
\mathcal E(u &,v)
= -\int_X [BD_Hu(x),D_Hv(x)]_H \nu_\infty(dx) \\
= & - \sum_{n=1}^m\int_X[D_Hv(x),Bi^*x^*_n]_H\frac{\partial \varphi}{\partial \xi_n}(\langle x, x^*_1\rangle_{X\times X^*},\ldots,\langle x,x^*_m\rangle_{X\times X^*}) \nu_\infty(dx) \\
=& \sum_{n=1}^m\int_X v(x)\Big(\sum_{j=1}^m\frac{\partial^2 \varphi}{\partial \xi_n\partial \xi_j}[i^*x^*_j,Bi^*x^*_n]_H
- \frac{\partial \varphi}{\partial \xi_n}(\langle x, x^*_1\rangle_{X\times X^*},\ldots,\langle x,x^*_m\rangle_{X\times X^*})\widehat{V^*Bi^*x^*_n}(x) \\
& \qquad \qquad -[D_HU(x),Bi^*x^*_n]_H  \frac{\partial \varphi}{\partial \xi_n}(\langle x, x^*_1\rangle_{X\times X^*},\ldots,\langle x,x^*_m\rangle_{X\times X^*})
\Big)  \nu_\infty(dx) \\
=&- \int_X v(x)\Big(\frac12{\rm Tr}[D_H^2 u(x)]_H+\langle x,A^*Du(x)\rangle_{X\times X^*}  +[BD_Hu(x),D_HU(x)]_H\Big)  \nu_\infty(dx).
\end{align*}
Since
\begin{align*}
x\mapsto \frac12{\rm Tr}[D_H^2u(x)]_H+\langle x,A^*Du(x)\rangle_{X\times X^*} +[BD_Hu(x),D_HU(x)]_H\in \elle^2(X,\nu_\infty),
\end{align*}
it follows that $u\in D(L_2)$ and
\begin{align*}
L_2u(x)=\frac12{\rm Tr}[D_H^2u(x)]_H+\langle x,A^*Du(x)\rangle_{X\times X^*} +[BD_Hu(x),D_HU(x)]_H, \quad \nu_\infty{\textup{-a.e.} }\ x\in X.
\end{align*}
\end{proof}

\subsection{The nonsymmetric Ornstein-Uhlenbeck operator in \texorpdfstring{$\elle^p(X,\nu_\infty)$}{Lp}}
In this subsection we consider the realization of the semigroup $(T_2(t))_{t\geq0}$ in $\elle^p(X,\nu_\infty)$ with $p\in(1,+\infty)$, and we show some important properties of the perturbed Ornstein-Uhlenbeck semigroup in $\elle^p(X,\nu_\infty)$. We need a technical lemma, which is the analogous of \cite[Lemma 2.7]{DPL14} in our setting, about the differentiability of the positive and negative part of a function $u\in W^{1,2}_H(X,\nu_\infty)$.

\begin{lemma}
\label{diff_modulo_parte_pos_neg}
Let $u\in W^{1,2}_H(X,\nu_\infty)$. Then, $|u|, u^+,u^-\in W^{1,2}_H(X,\nu_\infty)$ and $D_H|u|=\sign (u)D_Hu$. Further, $D_Hu$ vanishes on $u^{-1}(0)$ $\nu_\infty$-a.e.; $D_H(u^+)=\mathds 1_{\{u>0\}}D_Hu$ and $D_H(u^-)=-\mathds 1_{\{u<0\}}D_Hu$.
\end{lemma}
\begin{proof}
The proof is analogous to that of \cite[Lemma 2.7]{DPL14} and we omit it. We simply remark that, to prove that second part, as in the proof of Proposition \ref{clos_D_H_peso} we consider the basis $\{v_n:n\in\N\}$ of $H$ of elements of $D(V^*)$ and we show that
\begin{align*}
\int_{\{u=0\}}[D_Hu,v_i]_H\varphi d\nu_\infty=0,
\end{align*}
for any $u\in W^{1,2}_H(X,\nu_\infty)$ and any $\varphi\in \fcon_b^{1}(X)$.
\end{proof}

Thanks to Lemma \ref{diff_modulo_parte_pos_neg} we can prove that both $L_2$ and $L_2^*$ are Dirichlet operators and therefore that both $(T_2(t))_{t\geq0}$ and $(T_2^*(t))_{t\geq0}$ are sub-Markovian operators on $\elle^2(X,\nu_\infty)$. For reader's convenience, we recall the definitions of  Dirichlet and sub-Markovian operators and their main properties (see e.g. \cite[Chapter 1, Definition 4.1 \& Proposition 4.3]{MR92}).
\begin{defn}
Let $(E,B,\mu)$ be a measure space and let $\mathscr H:=\elle^2(E,\mu)$ be a Hilbert space.
\begin{itemize}
\item[(i)] A semigroup $(S(t))_{t\geq0}$ on $\mathscr H$ is called sub-Markovian if for any $t\geq0$ and any $f\in \mathscr H$ with $0\leq f\leq 1$ $\mu$-a.e., we have $0\leq S(t)f\leq 1$ $\mu$-a.e.
\item[(ii)] A closed linear densely defined operator $A$ on $\mathscr H$ is called Dirichlet operator on $\mathscr H$ if
\begin{align*}
\int_E Au(u-1)^+d\mu\leq 0, \quad u\in D(A).
\end{align*}
\end{itemize}
\end{defn}

\begin{pro}
Let $(S(t))_{t\geq0}$ be a strongly continuous contraction semigroup on $\elle^2(E,\mu)$ with generator $\mathcal A$. Then, the following are equivalent:
\begin{itemize}
\item[(i)] $(S(t))_{t\geq0}$ is a sub-Markovian semigroup on $\elle^2(E,\mu)$.
\item[(ii)] $\mathcal A$ is a Dirichlet operator on $\elle^2(E,\mu)$.
\end{itemize}
\end{pro}

We prove that it is possible to extend the semigroup $(T_2(t))_{t\geq0}$ to a strongly continuous contraction semigroup on $\elle^p(X,\nu_\infty)$ for any $p\in[1,+\infty)$. We follow the proof of \cite[Theorem 1.4.1]{DAV89}.
\begin{pro}
\label{generazione_smgr_Lp}
The semigroup $(T_2(t))_{t\geq0}$ can be uniquely extended to a positive contraction semigroup $(T_p(t))_{t\geq0}$ on $\elle^p(X,\nu_\infty)$ for any $p\in[1,+\infty)$. These semigroups are strongly continuous and are consistent in the sense that if $q\geq p$ then $T_p(t)f=T_q(t)f$ for any $f\in \elle^q(X,\nu_\infty)$.
\end{pro}
\begin{proof}
For reader's convenience, we split the proof into different steps.

\vspace{2mm}
{\bf Step $1$}. We prove that both $L_2$ and $L_2^*$ are Dirichlet operators on $\elle^2(X,\nu_\infty)$. Let $u\in D(L_2)$. Then, $u\in W^{1,2}_H(X,\nu_\infty)$ and from Lemma \ref{diff_modulo_parte_pos_neg} we infer that $(u-1)^+\in W^{1,2}_H(X,\nu_\infty)$ and $D_H(u-1)^+=\mathds 1_{u\geq 1}D_Hu$. Therefore,
\begin{align*}
\int_XL_2u(u-1)^+d\nu_\infty
= & \int_X[BD_Hu,D_H(u-1)^+]_Hd\nu_\infty
=\int_{\{u>1\}}[BD_Hu,D_Hu]_Hd\nu_\infty\leq 0,
\end{align*}
thanks to Lemma \ref{propr_B}. The computations for $L_2^*$ are analogous. Hence, both $L_2$ and $L_2^*$ are Dirichlet operators on $\elle^2(X,\nu_\infty)$, which means that $(T_2(t))_{t\geq0}$ and $(T_2^*(t))_{t\geq0}$ are sub-Markovian semigroups on $\elle^2(X,\nu_\infty)$.

\vspace{2mm}
{\bf Step $2$}. We claim that $\elle^1(X,\nu_\infty)$ and  $\elle^\infty(X,\nu_\infty)$ are invariant for $T_2(t)$, for any $t\geq0$. From Step $1$ we know that for any $f\in \elle^2(X,\nu_\infty)$ such that $0\leq f\leq 1$ $\nu_\infty$-a.e.we have $0\leq T_2(t)f\leq 1$ $\nu_\infty$-a.e. Then, it follows that $\elle^\infty(X,\nu_\infty)$ is invariant under $(T_2(t))_{t\geq0}$. Obviously, the same holds true for $(T_2^*(t))_{t\geq0}$. Let $f\in \elle^2(X,\nu_\infty)$. For any $g\in \elle^\infty(X,\nu_\infty)$, we have
\begin{align}
\label{stima_L1_smgr}
\left|\int_XT_2(t)f gd\nu_\infty\right|
= & \left|\int_XfT_2^*(t)gd\nu_\infty\right|
\leq \|f\|_{\elle^1(X,\nu_\infty)}\|g\|_{\elle^\infty(X,\nu_\infty)}, \quad t\geq0,
\end{align}
since also $T_2^*(t)$ is a contraction on ${\elle^\infty(X,\nu_\infty)}$. \eqref{stima_L1_smgr} and the density of $\elle^2(X,\nu_\infty)$ in $\elle^1(X,\nu_\infty)$ implies that for any $f\in\elle^1(X,\nu_\infty)$ we have $T_2(t)f\in\elle^1(X,\nu_\infty)$ for any $t\geq0$ and
\begin{align*}
\|T_2(t)f\|_{\elle^1(X,\nu_\infty)}\leq \|f\|_{\elle^1(X,\nu_\infty)}, \quad t\geq0.
\end{align*}
The claim is so proved. By applying the Riesz-Thorin Interpolation Theorem \cite[Section 1.18.7, Theorem 1]{TR78} we conclude that $(T_2(t))_{t\geq0}$ extends to a positive contraction semigroup $(T_p(t))_{t\geq0}$ on $\elle ^p(X,\nu_\infty)$ for any $p\in[1,+\infty)$. Uniqueness follows by density.

\vspace{2mm}
{\bf Step $3$}. Now we show that $(T_p(t))_{t\geq0}$ is strongly continuous if $p\in[1,+\infty)$. Let $f\in C_b(X)$. We have
\begin{align*}
\lim_{t\rightarrow0}\|T_1(t)f-f\|_{\elle^1(X,\nu_\infty)}
= & \lim_{t\rightarrow0}\int_X|T_1(t)f-f|d\nu_\infty
\leq \lim_{t\rightarrow0}\nu_\infty(X)^{1/2}\|T_2(t)f-f\|_{\elle^2(X,\nu_\infty)}=0.
\end{align*}
The density of continuous bounded functions in $\elle^1(X,\nu_\infty)$ implies that $(T_1(t))_{t\geq0}$ is strongly continuous on $\elle^1(X,\nu_\infty)$. By interpolation, we infer the strong continuity of $(T_p(t))_{t\geq0}$ on $\elle^p(X,\nu_\infty)$ for any $p\in(1,2)$. Finally, the reflexivity of $\elle^p(X,\nu_\infty)$ (see e.g. \cite[Section 4, Theorem 1]{DU77}) for any $p\in(1,+\infty)$ and \cite[Theorem 1.34]{DAV80} allow us to conclude that $(T_p(t))_{t\geq0}$ is strongly continuous on $\elle^p(X,\nu_\infty)$ for any $p\in(2,+\infty)$. 
\end{proof}

For any $p\in[1,+\infty)$ let us denote by $L_p$ the infinitesimal generator of $(T_p(t))_{t\geq0}$. Since $(T_p(t))_{t\geq0}$ is a positive strongly continuous semigroup for any $p\in[1,+\infty)$, we get $1\in\rho(L_p)$ for any $p\in[1,+\infty)$. The following result holds true.

\begin{pro}
\label{inc_dom_oper_Lp}
For any $p,q\in(1,+\infty)$ with $q>p$, we have $D(L_q)\subset D(L_p)$ with continuous embedding and for any $u\in D(L_q)$ we have that $L_qu=L_pu$. In particular, $D(L_p)\subset W^{1,2}_H(X,\nu_\infty)$ with continuous embedding for any $p\geq 2$.
\end{pro}
\begin{proof}
Let $u\in D(L_q)$. Then, we have
\begin{align*}
\|t^{-1}(T_p(t)u-u)-L_qu\|^p_{\elle^p(X,\nu_\infty)}
=&  \int_X\left|\frac{T_q(t)u-u}{t}-L_qu\right|^pd\nu_\infty \\
\leq & (\nu_\infty(X))^{1/r'}\|t^{-1}(T_q(t)u-u)-L_qu\|_{\elle^q(X,\nu_\infty)}^{1/r}\rightarrow0,
\end{align*}
as $t\rightarrow0$, where $r=\frac qp$ and $r'=\frac{q}{q-p}$ . Hence, $u\in D(L_p)$ and $L_pu=L_qu$.

The last part follows from the fact that $D(L_2)\subset W^{1,2}_H(X,\nu_\infty)$ with continuous injection.
\end{proof}

\section{Analyticity of the semigroup associated to \texorpdfstring{$L_p$}{Lp}}
\label{analyticity}
In this section we show that $L_p$ is sectorial in $\elle^p(X,\nu_\infty)$ for any $p\in(1,+\infty)$, i.e., $(T_p(t))_{t\geq0}$ is an analytic semigroup on the sector ${\Sigma_{\theta_p}}:=\left\{re^{i\phi}:r>0, |\phi|< \theta_p\right\}$, where 
\begin{align}
\label{angolo_settore}
{\rm cotg}(\theta_p)=\frac{\sqrt{(p-2)^2+p^2\gamma^2}}{2\sqrt{p-1}}, \quad \gamma:=\|B-B^*\|_{\mathcal L(H)}.
\end{align}
To this aim we follow the approach of \cite[Section 3]{MV07}. We introduce the following spaces of functions.
\begin{defn}
For any $p\in(1,+\infty)$ we set $\elle^p_{\C}(X,\nu_\infty):=\elle^p(X,\nu_\infty)+i\elle^p(X,\nu_\infty)$ with dual product $(f,g):=\int_Xf\overline gd\nu_\infty$ for any $f\in \elle^p_{\C}(X,\nu_\infty)$ and $g\in \elle^{p'}_{\C}(X,\nu_\infty)$. For any $k\in\N\cup\{\infty\}$ we denote by $\fcon_{b}^{k,1}(X;\C)$ the functions $f=u+iv$ such that $u,v\in \fcon_{b}^{k,1}(X)$. We set $W^{1,p}_{H,\C}(X,\nu_\infty):=W^{1,p}_{H}(X,\nu_\infty)+iW^{1,p}_{H}(X,\nu_\infty)$ for any $p\in(1,+\infty)$.

We consider the operator $L_p^{\C}$, on $D(L_p^{\C}):=D(L_p)+iD(L_p)$ endowed with the complexified norm of $D(L_p)$, defined by $L_p^{\C}f:=L_pu+iL_pv$, where $f:=u+iv\in D(L_p^{\C})$.
\end{defn}

\begin{rmk}
\label{estensione_ris}
It is not hard to prove that all the results in Section \ref{ferragosto} and Section \ref{nonsymm_OU_op_section} can be extended by complexification to the complex case.
\end{rmk}

\begin{rmk}
\label{duality_set}
We recall the definition of {\it duality map}. Given a Banach space $Y$ and given a duality $( \cdot,\cdot)_{Y\times Y^*}$ between $Y$ and $Y^*$, the duality map $\partial (y)\subset Y^*$ of $y\in Y$ is given by $\partial(y):=\{y^*\in Y^*:(y,y^*)_{Y\times Y^*}=\|y\|^2_Y=\|y^*\|^2_{Y^*}\}$. For any $p\in(1,+\infty)$ and any $f\in \elle^p_{\C}(X,\nu_\infty)$, with respect to the duality $\langle f,g\rangle:=\int_Xfgd\nu_\infty$, we have $\partial (f)=\{\|f\|_{p}^{2-p}f^*\}$, with 
\begin{align*}
f^*(x):=
\begin{cases}
\overline f(x)|f(x)|^{p-2}, & f(x)\neq0, \\
0, & f(x)=0.
\end{cases}
\end{align*}
In particular, $f^*$ is well defined also for $p\in(1,2)$.
\end{rmk}

For any $\theta\in[0,\pi/2)$ we set $C_\theta:={\rm cotg}(\theta)$. We will apply the following proposition, which is an adaptation of \cite[Proposition 3.2]{MV07} to our situation.
\begin{pro}
\label{vasca}
Let $\mathscr A$ be a densely defined operator on $\elle^p(X,\nu_\infty)$ and assume that $1\in \rho(\mathscr A)$. Then, the following are equivalent:
\begin{itemize}
\item[(i)] $\mathscr A$ generates an analytic $C_0$-semigroup on $\elle^p(X,\nu_\infty)$ which is contractive on $\Sigma_\theta$;
\item[(ii)] for any $f\in D(\mathscr A)$ we have 
\begin{align}
\left|{\rm Im}\left(\int_X\mathscr A f f^* d\nu_\infty\right)\right|\leq -C_\theta{\rm Re}\left(\int_X\mathscr A f  f^* d\nu_\infty\right). \label{moscone}
\end{align}
\end{itemize}
\end{pro}

\begin{rmk}
\label{diff_dual_funct}
Let $f\in\fcon_{b}^{1}(X;\C)$ and let $p\geq2$. Then, $f^*\in W^{1,2}_H(X,\nu_\infty)$ and we have
\begin{align*}
D_Hf^*
= & D_H(\overline f|f|^{p-2})
=|f|^{p-2}D_H\overline f+(p-2)|f|^{p-4}f \overline fD_Hf,
\end{align*}
where $f=u+iv$. In particular, $D_Hf^*$ is bounded. It is enough to consider the sequence $(f_n)\subset \fcon_b^1(X)$ given by $f_n:=\overline f(\theta_n\circ f)$, with $\theta_n(\xi)=\left(\xi^2+\frac1n\right)^{(p-2)/2}$ for any $\xi\in\R$ and $n\in\N$.
\end{rmk}

Finally, we recall \cite[Lemma 3.3]{MV07}, which is obtained by repeating the computations of \cite[Lemma 5]{CFMP05}.
\begin{lemma}
For any $f\in\fcon_{b}^{1}(X;\C)$ and any $p\in[2,+\infty)$ we have
\begin{align}
-{\rm Re}[BD_Hf,D_H{ f^*}]_H
= & -{\rm Re}[B^*D_Hf,D_H {f^*}]_H \notag \\
=&  \frac12|f|^{p-4}\left((p-1)|{\rm Re}(\overline f D_Hf)|_H^2+|{\rm Im}(\overline f D_Hf)|_H^2\right),
\label{parte_Re}
\end{align}
and
\begin{align}
\label{parte_Im1}
{\rm Im}[BD_Hf,D_H {f^*}]_H
& =p|f|^{p-4}\left[\left(B+\frac12I_H\right){\rm Im}(\overline f D_Hf),{\rm Re}(\overline f D_Hf)\right],\\
\label{parte_Im2}
{\rm Im}[B^*D_Hf,D_H {f^*}]_H
& =p|f|^{p-4}\left[\left(B^*+\frac12I_H\right){\rm Im}(\overline f D_Hf),{\rm Re}(\overline f D_Hf)\right].
\end{align}
\end{lemma}

Following the arguments of \cite[Theorem 3.4]{MV07} we obtain the analyticity of the semigroup $(T_p(t))_{t\geq0}$ for any $p\in(1,+\infty)$.
\begin{pro}
\label{marola}
$(T_p(t))_{t\geq0}$ is analytic in $\elle^p(X,\nu_\infty)$ on the sector $\Sigma _{\theta_p}$.
\end{pro}
\begin{proof}
We show that Proposition \ref{vasca}$(ii)$ is satisfied with $\mathscr A=L_p$ and $\theta=\theta_p$. To begin with, the positivity of $(T_p(t))_{t\geq0}$ implies that $1\in \rho(L_p)$ for any $p\in(1,+\infty)$. At first we consider $p\in[2,+\infty)$ and then we deal with the case $p\in(1,2)$.

\vspace{2mm}
{\bf Step $1$}. Let $p\in[2,+\infty)$, let $f\in \fcon_{b}^{2,1}(X;\C)$ and let $f^*:=\overline f|f|^{p-2}\in C_b(X)$. Let us set
\begin{align*}
a:=|{\rm Re}(\overline {f} D_Hf)|_H, \quad 
b:=|{\rm Im}(\overline {f} D_Hf)|_H.
\end{align*}
From \eqref{parte_Re} we infer that
\begin{align}
\label{parte_Re_m}
-{\rm Re}[BD_Hf,D_H{f^*}]_H
= \frac12|f|^{p-4}\left((p-1)a^2+b^2\right).
\end{align}
Since $B+B^*=-I_H$ we easily get
\begin{align}
\label{stima_gamma}
\left|B+\frac12I_H\right|_{\mathcal L(H)}
= & \left|\frac12B-\frac12B^*\right|_{\mathcal L(H)}
= \frac14\gamma^2+\left(\frac12-\frac1p\right)^2,
\end{align}
where $\gamma$ has been introduced in \eqref{angolo_settore}. The Cauchy-Schwarz inequality and \eqref{parte_Im1} give
\begin{align}
\label{parte_Im_m}
|{\rm Im}[BD_Hf,D_H{f^*}]_H|
\leq &  |f|^{p-4}C_{\theta_p}ab\sqrt{p-1}.
\end{align}
Thanks to the Young's inequality $2ab\sqrt{p-1}\leq (p-1)a^2+b^2$ we deduce that
\begin{align}
\label{stima_parte_im_parte_re}
|{\rm Im}[BD_Hf,D_H{f^*}]_H|
\leq &  \frac12|f|^{p-4}C_{\theta_p}\left((p-1)a^2+b^2\right)
=-{\rm Re}[BD_Hf,D_H {f^*}]_H,
\end{align}
for any $f\in \fcon_b^{2,1}(X)$.

%Then, from \eqref{stima_parte_im_parte_re} we infer
%\begin{align*}
%%\left|{\rm Im}\left(\int_XL_p f f^* d\nu_\infty\right)\right| 
% \left|{\rm Im}\left(\int_XL_2 f f^* d\nu_\infty\right)\right|
%=  & \left|{\rm Im}\left(\int_X[BD_Hf,D_H {f^*}]_H d\nu_\infty\right)\right| \\
%\leq & -C_{\theta_p} \int_X{\rm Re}[BD_Hf,D_H {f^*}]_Hd\nu_\infty  \\
%= & -C_{\theta_p} \ \!{\rm Re}\left(\int_X L_2f f^* d\nu_\infty\right).
%%= & -C_{\theta_p} \ \!{\rm Re}\left(\int_X L_pf f^* d\nu_\infty\right).
%\end{align*}
%Hence, Proposition \ref{vasca}$(ii)$ holds true for any $f\in \fcon_{b}^{2,1}(X;\C)$. 
Let $f=u+iv\in D(L_p^{\C})$ and let us consider a sequence $(f_n:=u_n+iv_n)\subset \fcon_{b}^{2,1}(X;\C)$ such that $u_n\rightarrow u$ and $v_n\rightarrow v$ in $W^{1,2}_H(X,\nu_\infty)$, and $u_n\rightarrow u$ and $v_n\rightarrow v$ $\nu_\infty$-a.e. in $X$. These sequences exists thanks to Remark \ref{eq_chiusura_dom_C2}, to Proposition \ref{inc_dom_oper_Lp} and thanks to Remark \ref{estensione_ris}. 

From the definition of $f_m^*$, we have that $f_m^*\rightarrow f^*$ $\nu_\infty$-a.e. in $X$. Further, $\|f^*_m\|_{\elle^{p'}(X,\nu_\infty)}=\|f_m\|_{\elle^p(X,\nu_\infty)}$ is uniformly bounded with respect to $m\in\N$. Hence, there exists a function $g\in \elle^{p'}(X,\nu_\infty)$ such that, up to a subsequence which we still denote by $(f^*_m)$, $f^*_m\rightharpoonup g$ as $m\rightarrow+\infty$ in $\elle^{p'}(X,\nu_\infty)$. Since $f_m^*\rightarrow f^*$ $\nu_\infty$-a.e. in $X$, it follows that $g=f^*$ $\nu_\infty$-a.e. in $X$, i.e.,
\begin{align}
\label{conv_funz_compl}
\int_X hf_m^*d\nu_\infty\rightarrow \int_X hf^*d\nu_\infty, \quad n\rightarrow+\infty, \quad \forall h\in \elle^p(X,\nu_\infty).
\end{align}
From Remark \ref{diff_dual_funct} it follows that
\begin{align}
& \lim_{m\rightarrow+\infty}\int_X\left|{\rm Re}[BD_Hf_m,D_H{f_m^*}]_H
-{\rm Re}[BD_Hf,D_H{f_m^*}]_H\right|d\nu_\infty=0, \label{conv_parte_re}\\
& \lim_{m\rightarrow+\infty}\int_X\left|{\rm Im}[BD_Hf_m,D_H{f_m^*}]_H
-{\rm Im}[BD_Hf,D_H{f_m^*}]_H\right|d\nu_\infty=0.  \label{conv_parte_im}
\end{align}
Indeed,
\begin{align*}
&  \int_X\left|{\rm Re} [BD_Hf_m,D_H{f_m^*}]_H-{\rm Re}[BD_Hf,D_H{f_m^*}]_H\right|d\nu_\infty\\
&\quad \quad\quad \quad \quad  \leq  \|B\|_{\mathcal L(H)}\int_X|D_Hf-D_Hf_m|_H|D_Hf^*_m|_Hd\nu_\infty \\
&\quad \quad\quad \quad \quad \leq \|B\|_{\mathcal L(H)}\|D_Hf_m-D_Hf\|_{\elle^p_{\C}(X,\nu_\infty;H)}\|D_Hf^*_m\|_{\elle^{p'}_{\C}(X,\nu_\infty;H)}.
\end{align*}
We claim that $\|D_Hf^*_m\|_{\elle^{p'}_{\C}(X,\nu_\infty;H)}$ is uniformly bounded with respect to $m\in\N$. Indeed, for any $m\in\N$ we have
\begin{align*}
\|D_Hf^*_m\|_{\elle^{p'}_{\C}(X,\nu_\infty;H)}^{p'}
\leq 2^{p'-1}\left(\int_X|f_m|^{p'(p-2)}|D_H\overline {f_m}|^{p'}d\nu_\infty+
(p-2)\int_X|f_m|^{p'(p-2)}|D_Hf_m|_H^{p'}d\nu_\infty\right).
\end{align*}
We recall that $p'=\frac{p}{p-1}$. By applying the H\"older inequality with $q=p-1$ and $q'=\frac{p-1}{p-2}$, it follows that
\begin{align*}
\|D_Hf^*_m\|_{\elle^{p'}_{\C}(X,\nu_\infty;H)}^{p'}
\leq 2^{p'-1}(p-1)\|f_m\|_{\elle^p_{\C}(X,\nu_\infty)}^{1/q'}\|D_H{f_m}\|_{\elle^p_{\C}(X,\nu_\infty;H)}^{1/q}\leq c_p, \quad m\in\N,
\end{align*}
for some positive constant $c_p$, since both $\|f_m\|_{\elle^p_{\C}(X,\nu_\infty)}$ and $\|D_H{f_m}\|_{\elle^p_{\C}(X,\nu_\infty;H)}$ converge as $n\rightarrow+\infty$. Then, the claim is true and \eqref{conv_parte_re} and \eqref{conv_parte_im} follow from the fact that $D_Hf_m\rightarrow D_Hf$ in $W^{1,2}_{H,{\C}}(X,\nu_\infty)$ as $m\rightarrow+\infty$. Same arguments also work for \eqref{conv_parte_im}.

From Proposition \ref{inc_dom_oper_Lp}, \eqref{stima_parte_im_parte_re}, \eqref{conv_funz_compl}, \eqref{conv_parte_re} and \eqref{conv_parte_im} we get
\begin{align}
\notag
\left|{\rm Im}\left(\int_XL_p ff^* d\nu_\infty\right)\right|
= &\left|{\rm Im}\left(\int_XL_2 f f^* d\nu_\infty\right)\right| 
=  \lim_{m\rightarrow+\infty}\left|{\rm Im}\left(\int_XL_2 f f_m^* d\nu_\infty\right)\right|  \notag \\
= &  \lim_{m\rightarrow+\infty}\left|\left(\int_X{\rm Im}[BD_Hf,D_H {f^*_m}]_H d\nu_\infty\right)\right| \notag \\
= &\lim_{m\rightarrow+\infty}\left|\left(\int_X{\rm Im}[BD_Hf_m,D_H {f_m^*}]_H d\nu_\infty\right)\right| \notag\\
\leq & -C_{\theta_p} \lim_{m\rightarrow+\infty}\int_X{\rm Re}[BD_Hf_m,D_H {f_m^*}]_Hd\nu_\infty \notag \\ 
\leq & -C_{\theta_p} \lim_{m\rightarrow+\infty}\int_X{\rm Re}[BD_Hf,D_H {f_m^*}]_Hd\nu_\infty \notag \\ 
= & -C_{\theta_p}\lim_{m\rightarrow+\infty} {\rm Re}\left(\int_X L_2f f^*_m d\nu_\infty\right)
= -C_{\theta_p}{\rm Re}\left(\int_X L_2f f^* d\nu_\infty\right) \notag \\
= & -C_{\theta_p}{\rm Re}\left(\int_X L_pf f^* d\nu_\infty\right).\label{stima_L_p>2}
\end{align}
This shows that Proposition \ref{vasca}$(ii)$ holds true for any $f\in D(L_p^{\C})$, for any $p\in[2,+\infty)$. 

\vspace{2mm}
{\bf Step $2$}.
Let $p\in(1,2)$. We claim that $D(L_2^{\C})$ is a core for $D(L_p^{\C})$. Remark \ref{inc_dom_oper_Lp} with $q=2$ implies that $D(L_2^{\C})\subset D(L_p^{\C})$. From Step $1$, we know that $(T_2(t))_{t\geq0}$ is analytic in $\elle^2(X,\nu_\infty)$ and therefore $T(t)D(L_2)\subset D(L_2)$ for any $t\geq0$. Since $T_p(t)=T_2(t)$ on $\elle^2(X,\nu_\infty)$, we infer the $T_p(t)D(L_2)=T_2(t)D(L_2)\subset D(L_2)$. Moreover, $\fcon_b^{2,1}(X)\subset D(L_2)$. This implies that $D(L_2)$ is dense in $\elle^p(X,\nu_\infty)$. From \cite[Chapter 1, Proposition 1.7]{EN00} and Remark \ref{estensione_ris} we deduce that the claim is true.

Let $f\in D(L_p^{\C})$ and let $(f_n)\subset D(L_2^{\C})$ be a sequence which converges to $f$ in $D(L_p^{\C})$ as $n\rightarrow+\infty$ and $f_n\rightarrow f$ $\nu_\infty$-a.e. in $X$. As in \eqref{conv_funz_compl}, we can prove that, up to a subsequence, $f_n^*\rightharpoonup f^*$ as $n\rightarrow+\infty$ in $\elle^{p'}(X,\nu_\infty)$. Then, we have
\begin{align}
\label{passaggio_1<2}
\left|{\rm Im}\int_XL_pf f^*d\nu_\infty\right|
=\lim_{n\rightarrow+\infty} \left|{\rm Im}\int_XL_pf_n f_n^*d\nu_\infty\right|
= \lim_{n\rightarrow+\infty} \left|{\rm Im}\int_XL_2f_n f_n^*d\nu_\infty\right|
\end{align}
and the last equality follows from Proposition \ref{inc_dom_oper_Lp} with $q=2$. From \eqref{stima_L_p>2} with $p=2$ and $f$ replaced by $f_n$ we infer that
\begin{align}
\label{passaggio2<2}
\left|{\rm Im}\int_XL_2f_n f_n^*d\nu_\infty\right|
\leq -C_{\theta_p}{\rm Re}\left(\int_X L_2f_n f_n^* d\nu_\infty\right), \quad n\in\N.
\end{align}
Collecting \eqref{passaggio_1<2} and \eqref{passaggio2<2} we get
\begin{align*}
\left|{\rm Im}\int_XL_pf f^*d\nu_\infty\right|
\leq & -\lim_{n\rightarrow+\infty}C_{\theta_p}{\rm Re}\left(\int_X L_2f_n f_n^* d\nu_\infty\right)
=-\lim_{n\rightarrow+\infty}C_{\theta_p}{\rm Re}\left(\int_X L_pf_n f_n^* d\nu_\infty\right) \\
= & -C_{\theta_p}{\rm Re}\left(\int_X L_pf f^* d\nu_\infty\right).
\end{align*}
This implies that Proposition \ref{vasca}$(ii)$ is satisfied for $p\in(1,2)$.

\end{proof}

\section{Example}
\label{example}

We provide an example of operators $A$ and $Q$ which satisfy Hypotheses \ref{ipo_1}, \ref{portafoglio} and \ref{ipo_RKH}. Let $X:=\elle^2(0,1)$, let $A$ be the realization of the Laplace operator in $\elle^2(0,1)$ with Dirichlet boundary conditions and domain $W^{2,2}((0,1),d\xi)\cap W^{1,2}_0((0,1),d\xi)$, and let $Q:X\rightarrow X$ be the covariance operator of the Wiener measure on $X$, i.e.,
\begin{align}
\label{op_Q_ex}
Qf(x):=\int_0^1\min\{x,y\}f(y)dy, \quad x\in(0,1),
\end{align}
for any $f\in \elle^2(0,1)$ (see e.g. \cite{US95}). It is well known that $A$ is self-adjoint and that $e_k=\sqrt 2\sin(k\pi\cdot)$, $k\in\N$, is an orthonormal basis of $\elle^2((0,1),d\xi)$ of eigenvectors of $A$ with corresponding eigenvalues $\lambda_k=-k^2\pi^2$. We denote by $(e^{tA})_{t\geq0}$ the semigroup generated by $A$. $(e^{tA})_{t\geq0}$ is analytic on $\elle^2((0,1),d\xi)$ and $e^{tA}e_k=e^{-k^2\pi^2t}e_k$ for any $k\in\N$. Then, it is not hard to see that for any smooth function $f$ we have
\begin{align*}
(Qe^{sA}f)(x)
= & \sqrt 2\sum_{k=1}^\infty e^{-k^2\pi^2s}\langle f,\sqrt 2\sin(k\pi\cdot)\rangle_{\elle^2}\left(\frac1{k^2\pi^2}\sin(k\pi x)+\frac{(-1)^{k+1}}{k\pi}x\right).
\end{align*}
Moreover,
\begin{align*}
(e^{sA}Qe^{sA}f)(x)
= & \sqrt 2\sum_{k=1}^\infty e^{-2k^2\pi^2s}\langle f,\sqrt 2\sin(k\pi\cdot)\rangle_{\elle^2}\frac1{k^2\pi^2}\sin(k\pi x)\\
& +2\sum_{k,j=1}^\infty e^{-(k^2+j^2)\pi^2s}\langle f,\sqrt 2\sin(k\pi\cdot)\rangle_{\elle^2}\frac{(-1)^{k+1}}{k\pi}
\langle x, \sqrt 2\sin(j\pi\cdot)\rangle_{\elle^2}\sin(j\pi x) \\
= & \sqrt 2\sum_{k=1}^\infty e^{-2k^2\pi^2s}\langle f,\sqrt 2\sin(k\pi\cdot)\rangle_{\elle^2}\frac1{k^2\pi^2}\sin(k\pi x)\\
& +2\sqrt 2\sum_{k,j=1}^\infty e^{-(k^2+j^2)\pi^2s}\langle f,\sqrt 2\sin(k\pi\cdot)\rangle_{\elle^2}\frac{(-1)^{k+j+2}}{kj\pi^2}\sin(j\pi x).
\end{align*}
Integrating between $0$ and $t$ we get
\begin{align*}
(Q_t)f(x)
= & \sqrt 2\sum_{k=1}^\infty \langle f,e_k\rangle_{\elle^2}\frac{1-e^{-2k^2\pi^2t}}{2k^4\pi^4}\sin(k\pi x)\\
& +2\sqrt 2\sum_{k,j=1}^\infty\langle f,e_k\rangle_{\elle^2}\frac{(-1)^{k+j+2}(1- e^{-(k^2+j^2)\pi^2t})}{kj(k^2+j^2)\pi^4}\sin(j\pi x).
\end{align*}
\begin{pro}
$Q_t$ is a trace class operator for any $t>0$, $Q_t\rightarrow Q_\infty$ in the operator norm and $Q_\infty$ is a trace class operator, where
\begin{align}
Q_\infty f(x)
= & \sqrt 2\sum_{k=1}^\infty \langle f,e_k\rangle_{\elle^2}\frac{1}{2k^4\pi^4}\sin(k\pi x)
+2\sqrt 2\sum_{k,j=1}^\infty\langle f,e_k\rangle_{\elle^2}\frac{(-1)^{k+j+2}}{kj(k^2+j^2)\pi^4}\sin(j\pi x) \notag \\
= & \frac{3\sqrt 2}{2}\sum_{k=1}^\infty \langle f,e_k\rangle_{\elle^2}\frac{1}{2k^4\pi^4}\sin(k\pi x)
+2\sqrt 2\sum_{j\neq k}^\infty\langle f,e_k\rangle_{\elle^2}\frac{(-1)^{k+j+2}}{kj(k^2+j^2)\pi^4}\sin(j\pi x). \label{op_Q_infty_ex}
\end{align}
\end{pro}
\begin{proof}
From the above computations we have
\begin{align*}
\sum_{k=1}^\infty\langle Q_te_k,e_k\rangle_{\elle^2}
= & \frac{3\sqrt 2}2\sum_{k=1}^\infty\frac{1-e^{-k^2\pi^2t}}{k^4\pi^4}<+\infty,
\end{align*}
and
\begin{align*}
\sum_{k=1}^\infty\langle Q_\infty e_k,e_k\rangle_{\elle^2}
= & \frac{3\sqrt2}2\sum_{k=1}^\infty\frac{1}{k^4\pi^4}<+\infty.
\end{align*}
\end{proof}

Finally, let us take $U:X\rightarrow \R$ defined by
\begin{align*}
U(f):=\int_0^1f(\xi)^2d\xi, \quad f\in X.
\end{align*}

From \cite[Subection 7.1]{CF16} we infer that $U\in W^{1,p}_H(X,\mu_\infty)$ for any $p\in(1,+\infty)$. Hence, the Ornstein-Uhlenbeck operator $L_p$ is sectorial in $\elle^p(\elle^{2}(0,1),e^{-U}\mu_\infty)$ for any $p\in(1,+\infty)$.

%Let us explicitly compute $L_p$ on a class of functions. Let $\varphi\in C_b^2(\R)$ and let us set $f(g):=\varphi(g_k)$ for any $g\in \elle^2(0,1)$, where $\langle g,e_k\rangle_{\elle^2}$. It follows that $f\in \fcon_b^{2,1}(X)$ and that
%\begin{align*}
%& Df(g)=\varphi'(g_k)e_k, & &  D^2f(g)=\varphi''(g_k)e_k\otimes e_k, \\ 
%& ADf(g)=-\lambda_k\varphi'(g_k)e_k, &   &  Qe_k(x)=\sqrt2\ \!\frac{\sin(k\pi x)}{k^4\pi^4}+\sqrt2\ \!x\frac{(-1)^{k+1}}{k^2\pi^2}.
%\end{align*}
%We recall that $[D_HU(g),D_
%\begin{align*}
%L_pf(g)
%= & \frac12\varphi''(g_k)\langle Qe_k,Qe_k\rangle_{\elle^2}+\varphi'(g_k)\langle Ae_k,g\rangle_{\elle^2}+[D_Hf(g),D_HU(g)]_H \\
%& .
%\end{align*}

%\bibliographystyle{plain}
%\nocite{*} %\addcontentsline{toc}{section}{References}
%\bibliography{bib_analyticity.bib}
%\markboth{\textsc{References}}{\textsc{References}}
\end{document}